\documentclass[11pt,a4paper]{article}
\usepackage{exscale,times}
\usepackage{graphicx}
\usepackage{epstopdf}
\usepackage{amsmath,amsthm,amssymb}
\usepackage{epsfig}
\usepackage{latexsym}
\usepackage{subcaption}
\usepackage{amssymb}
\usepackage{float}
\usepackage{nicefrac} 
\usepackage{sidecap}
\usepackage{cite}
\usepackage{tikz,pgfplots}
\usepackage{hyperref}
\definecolor{winered}{rgb}{0.5,0,0}
\definecolor{darkgreen}{rgb}{0.0,0.5,0}
\hypersetup{
colorlinks=true,
linkcolor=winered,
citecolor=darkgreen
}
\usetikzlibrary{plotmarks}
\usepackage{placeins}
\usepackage{lipsum}
\usepackage{amsfonts}
\usepackage{graphicx}
\usepackage{epstopdf}
\usepackage{algorithmic}
\usepackage{amsopn}
\pgfplotsset{compat=1.10}
\newcommand{\xvec}{\mathbf{x}}
\newcommand{\Phivec}{\mathbf{\Phi}}
\newcommand{\vvec}{\mathbf{v}}
\newcommand{\wvec}{\mathbf{w}}
\newcommand{\vbar}{\bar{\vvec}}
\newcommand{\vhat}{\hat{\vvec}}
\newcommand{\evec}{\mathbf{e}}
\newcommand{\Lvec}{\mathbf{L}}
\newcommand{\Hvec}{\mathbf{H}}
\newcommand{\Tbar}{\overline{T}}
\newcommand{\Vbar}{\overline{V}}
\newcommand{\PH}{\mathrm{\Theta}}
\newcommand{\ph}{\mathrm{\theta}}
\newcommand{\ip}{\text{ip}}
\newcommand{\nip}{n_{\text{ip}}}
\newcommand{\R}[1]{\mathbb{R}^{#1}}
\newcommand{\diver}{\text{div}}
\newcommand{\refer}[1]{(\ref{#1})}
\newcommand{\ndof}{\text{ndof}}
\newcommand{\ndofv}{\mathrm{ndof}_\vvec}
\newcommand{\intrd}{\int\limits_{\R{3}}}
\DeclareMathOperator{\myspan}{span}
\newcommand\restr[2]{{
  \left.\kern-\nulldelimiterspace 
  #1 
  \vphantom{\big|} 
  \right|_{#2} 
  }}

\newtheorem{theorem}{Theorem}
\newtheorem{lemma}[theorem]{Lemma}

\newtheorem{remark}[theorem]{Remark}
\newtheorem{definition}[theorem]{Definition}

\title{A polynomial spectral method for the spatially homogeneous Boltzmann equation.}
\author{Gerhard Kitzler and Joachim Sch\"oberl}
\begin{document}

\maketitle

\begin{abstract}
\noindent
We present a spectral Petrov-Galerkin method for the Boltzmann collision operator. We expand the density distribution $f$ to high order orthogonal polynomials multiplied by a Maxwellian. By that choice, we can approximate on the whole momentum domain $\R3$ resulting in high accuracy at the evaluation of the collision operator. Additionally, the special choice of the test space naturally ensures conservation of mass, momentum and energy. By numerical examples we demonstrate the convergence (w.r.t. time) to the exact stationary solution. For efficiency we transfer between nodal and Maxwellian weighted Spherical Harmonics which are orthogonal w.r.t. the innermost integrals of the collision operator. Combined with efficient transformations between the bases and the calculation of the outer integrals this gives an algorithm of complexity $\mathcal{O}(N^7)$ and a storage requirement $\mathcal{O}(N^4)$ for the evaluation of the non linear Boltzmann collision operator. The presented method is applicable to a general class of collision kernels, among others including Maxwell molecules, hard and variable hard spheres molecules. Although faster methods are available, we obtain high accuracy even for very low expansion orders.
\end{abstract}



\section{The Boltzmann equation}
The Boltzmann equation is a statistical model for transport phenomena in a sufficiently dilute gas. In addition to gas dynamics, kinetic models as the Boltzmann equation have among others, attracted a wide range of applications including plasma physics, electron transport in semiconductors and also disciplines from biology. The unknown in kinetic equations is the density distribution in phase space, typically denoted by $f = f(t,\xvec,\vvec)$, such that
\begin{equation*}
\intrd f(t,\xvec,\vvec)\,d\vvec
\end{equation*}
gives the mass density at time $t$ and position $\xvec$. Similar, the momentum density, energy density and temperature are defined in terms of moments of $f$:
\begin{align*}
\rho(t,\xvec) &:= \intrd f(t,\xvec,\vvec)\,d\vvec \\
\mathbf{V}(t,\xvec)    &:= \frac{1}{\rho(t,\xvec}\intrd \vvec f(t,\xvec,\vvec)\,d\vvec \\
E(t,\xvec) &:= \frac{1}{2}\intrd |\vvec|^2f(t,\xvec,\vvec)\,d\vvec \\
T(t,\xvec)    &:= \frac{1}{3\rho(t,\xvec)}\intrd |\vvec-V(t,\xvec)|^2f(t,\xvec,\vvec)\,d\vvec.
\end{align*}
The Boltzmann equation which governs the time evolution of $f$ is
\begin{equation*}
\frac{\partial }{\partial t}f + \diver_\xvec(\vvec f) = Q(f),
\end{equation*}
where the Boltzmann collision operator $Q(f)(t,\xvec,\vvec) = Q(f(t,\xvec,\,.\,) )(\vvec)$ is
\begin{equation}
Q(f(t,\xvec,\,.\,) )(\vvec) := \intrd\int\limits_{S^2}B(\vvec,\wvec,\evec') [f(t,\xvec,\vvec')f(t,\xvec,\wvec')-f(t,\xvec,\vvec)f(t,\xvec,\wvec)]d\evec'd\wvec.
\label{eq:collop}
\end{equation}
Here, $\vvec' = \frac{\vvec+\wvec}{2}+\evec'\frac{|\vvec-\wvec|}{2}$ and $\wvec' = \frac{\vvec+\wvec}{2}-\evec'\frac{|\vvec-\wvec|}{2}$ are the pre collisional velocities. The collision operator is non linear, local in time and position, but global in velocity. It satisfies mass, momentum and energy conservation, expressed by the equations
\begin{equation}
\intrd Q(f)\left(\begin{matrix}1\\ \vvec \\ |\vvec|^2 \end{matrix}\right) = \mathbf{0}.
\label{eq:coll_invariants}
\end{equation}
The functions $1,\vvec,|\vvec|^2$ are consequently termed collision invariants. Also, each $\phi$ satisfying $\int Q(f)\phi\,d\vvec = 0$ is a linear combination of the basic collision invariants \cite{Cercignani_1994}. 
\par
Another important property of $Q$ we use in our numerical method is the kernel of the collision operator $\mathrm{ker}\,Q = \{f: f(t,\xvec,\vvec) = \rho(t,\xvec) e^{-\bigl|\frac{\vvec-\Vbar(t,\xvec)}{\sqrt{\Tbar(t,\xvec)}}\bigr|^2} \}$. In context of the Boltzmann equation these Gaussian peaks are termed Maxwellians. Note that $\mathrm{ker}\,Q$ also characterizes stationary solutions of the homogeneous Boltzmann equation 
\begin{equation}
\frac{\partial }{\partial t}f = Q(f).
\label{eq:bz_hom}
\end{equation}
The function $B(\vvec,\wvec,\evec') = B(|\vvec-\wvec|,\tfrac{(\vvec-\wvec)\cdot e'}{|\vvec-\wvec|})$ is called collision kernel. It is the probability density that two particles with velocity $\vvec'$ and $\wvec'$ collide and result in velocities $\vvec$ and $\wvec$. In many relevant cases, $B$ factorizes $B(\vvec,\wvec,\evec') = b_r(|\vvec-\wvec|)b_\theta(\tfrac{(\vvec-\wvec)\cdot \evec'}{|\vvec-\wvec|})$, with a power law for $b_r(|\vvec-\wvec|) = |\vvec-\wvec|^\beta$, with $\beta \in [0,1]$.
We refer to \cite{Cercignani_1994} for the mathematical theory of rarefied gas dynamics.
\par
Numerical methods have to tackle several problems. That is the high dimensionality: $(\xvec,\vvec) \in \R{3\times 3}$. The collision operator $Q$ has a quadratic non linearity and requires typically $\mathcal{O}(\ndof^3)$ operations for application, additionally its conservation properties shall hold on the discrete level.
\par
Due to the complexity and the high dimensionality, stochastic methods such as Direct Simulation Monte Carlo methods \cite{doi:10.1143/JPSJ.49.2050,MR845926,MR1352466,Rjasanow_2005} are attractive from a computational point of view. Turning to deterministic methods, discrete velocity models are among the most popular ones. Simplified speaking, particles are only allowed to have velocities on a finite grid. Consequently, the collision mechanism needs to be discretized in such a way, that pre and post collisional velocities are nodes of the grid, while maintaining the main physical properties of the collisional process. We refer to \cite{bobylevdvm1995,doi:10.1080/00411459608204829,Panferov_DVM}.

Widely used deterministic approaches are spectral methods \cite{Pareschi_fft_homo,MR1439069,Bobylev_hard_sphere,MR1756425,Ibragimov2002,MR2240637,GAMBA20092012,10.2307/43693916,MR2746671,doi:10.1137/16M1096001}. Usually the density distribution is approximated by a trigonometric Fourier expansion on some bounded domain. These methods benefit from the orthogonality of the trigonometric expansion, giving efficient algorithms for the collision operator. On the other hand, the integrals defining the collision operator, as well as the approximation domain have to be truncated at the expense of accuracy. The periodicity of the trigonometric basis combined with the larger support of $Q$ regarding the density distribution $f$ may additionally produce aliasing errors.
\par 
The current work is the natural extension from 2 velocity dimensions, discussed in  \cite{Kitzler2016,Kitzler2015} to 3. Closely related to these papers is the spectral method by Hiptmair et al. \cite{FGH14_563,GHP15_628}, being the extension of \cite{MR1701710} from radially symmetric solutions to general ones. This method was conducted in parallel to our 2d paper. To the best of our knowledge this approach is not yet available for three velocity dimensions.
\par
In a recent work by I. Gamba and S. Rjasanow \cite{GambaR2017} a similar expansion basis consisting of Spherical Harmonics and generalized Laguerre polynomials is constructed and used in a Petrov-Galerkin projection.  A different scaling in the polynomial argument results in a non diagonal mass matrix. Efficiency considerations for the collision operator are not presented therein. The extension to spatially inhomogeneous problems is treated in \cite{Torsten2017}.
Let us additionally state that in several situations, the non linear Boltzmann collision operator may be replaced by the BGK approximation, resulting in a simpler evolution equation for the unknown density $f$. We refer to \cite{cercignani1969mathematical} and the references therin for the BGK model. Numerical treatment of the BGK approximation can be found in \cite{Chen633}.

Our paper is specifically devoted to efficient application of the collision integrals. In contrast to spectral methods based on trigonometric expansion (which require a domain truncation), there is -- to the best of our knowledge -- no publication yet dealing with efficent evaluation of the collision operator in a polynomial spectral method on the unbounded velocity space. The presented algorithm impresses by low storage requirements of only $\mathcal{O}(N^4)$, where $N$ is the number of unknowns per direction, what is a lower storage requirement as in Fourier spectral methods \cite{doi:10.1137/16M1096001}. Calculations with 32 degrees of freedom for each Cartesian direction require less than 3 Megabytes. Combined with the good approximation properties of the trial space, lower expansion orders already yield highly accurate results (see sections \ref{subsec:maxwell_moments},\ref{subsec:hs_moments}) and therefore the higher numerical complexity of $\mathcal{O}(N^7)$ is compensated. In addition, using our algorithm for the evalution of the collision integral, matrix-matrix multiplications are performed almost exclusively, what enables the usage of highly optimized Lapack routines \cite{laug}. 

\section{A spectral projection for the homogeneous equation}
We consider the homogeneous Boltzmann equation \refer{eq:bz_hom} for some $f = f(t,\vvec)$. In the sequel our spectral Petrov Galerkin projection for that equation is presented. Let us first motivate the choice of the trial space. Since $Q$ is global in $\vvec$, we propose a global trial space. In order to have the kernel and consequently the stationary solution in the trial space we choose the following space for approximating the density $f$:
\begin{equation*}
V_{\Tbar,\Vbar,N}:=\{f \in L_2(\R3): f(\vvec) = e^{-\bigl|\frac{\vvec-\Vbar}{\sqrt{\Tbar}}\bigr|^2}p(\vvec):\,p \in P^N(\R3) \}.
\end{equation*}
$P^N(\R3)$ is the space of polynomials of total degree $N$. The parameters $\Vbar$ and $\Tbar$ are the macroscopic velocity and temperature of the Maxwellian in the expansion. If they are chosen in accordance to $f$'s mean velocity and temperature we expect excellent approximation properties. Due to the conservation properties of $Q$, $\Vbar$ and $\Tbar$ can be chosen a priori.
A motivation for the test space are the conservation properties of $Q$, which lead to $\frac{\partial}{\partial t}\rho(t,\xvec) = 0,\frac{\partial}{\partial t}\mathbf{V}(t,\xvec) = \mathbf{0},\frac{\partial}{\partial t}E(t,\xvec) = 0$ and $\frac{\partial}{\partial t}T(t,\xvec) = 0$. These conservations are naturally satisfied as soon as $1,\vvec,|\vvec|^2$ are in the test space. This leads to our test space, which is
\begin{equation*}
V_N:=P^N(\R3).
\end{equation*}
The Petrov-Galerkin projection now reads
\begin{equation}
\text{Find } f \in V_{\Tbar,\Vbar,N},\text{s.t.}\;\;\;\frac{\partial }{\partial t}\intrd f\phi\,d\vvec = \intrd Q(f)\phi\,d\vvec \quad \forall\, \phi \in V_N.
\label{eq:petrovgalerkin_hom}
\end{equation}
To obtain \refer{eq:coll_invariants} on the discrete level, the collision invariants need to be in $V_N$. Consequently $N=2$ is a minimum requirement.

\subsection{Polynomial bases in \texorpdfstring{$P^N(\R3)$ and $Q^N(\R3)$}{Polynomialbases} }
In this section we present the polynomial basis of $P^N(\R3)$ and $Q^N(\R3)$ we use. $Q^N(\R3)$ is the space of polynomials of partial degree $N$. To define the basis of $Q^N(\R3)$ we denote by $(v_\text{ip},\omega_\text{ip}),\,i=0\ldots N$ the nodes and weights of a Gauss-Hermite quadrature \cite{shen2011spectral} satisfying
\begin{equation*}
\int\limits_{\R{}} e^{-v^2}p(v) = \sum\limits_{\ip=0}^{N} \omega_\ip p(v_\ip),\;\forall\, p \in P^{2N+1}(\R{}).
\end{equation*}
\par
The basis polynomials are defined as Lagrange collocation polynomials to the Gauss-Hermite quadrature nodes and are denoted by $l$:
\begin{equation*}
l_j(v) := \prod\limits_{i=0 \atop i\neq j}^{N} \frac{v-v_i}{v_j-v_i}.
\end{equation*}
The three dimensional basis is constructed as the tensor product of the 1D polynomials and its elements are denoted by $L^{(N)}_j,j=0\dots \ndofv-1=(N+1)^3-1$:
\begin{equation*}
L^{(N)}_j(\vvec) = l_l(\vvec_1)l_m(\vvec_2)l_n(\vvec_3),
\end{equation*}
with $j = (N+1)^2l + (N+1)m+n$.
This is an orthogonal basis of $Q^N(\R3)$ resulting in a diagonal mass matrix:
\begin{equation}
\intrd e^{-|\vvec|^2}L^{(N)}_m(\vvec)L^{(N)}_n(\vvec)\,d\vvec = \sum\limits_{\ip=0}^{\ndofv-1}\omega^{(3)}_\ip L^{(N)}_m(\vvec_\ip^{(3)})L^{(N)}_n(\vvec_\ip^{(3)}) = \delta_{m,n}\omega^{(3)}_n
\label{eq:orthogonality_mass}
\end{equation}
In \refer{eq:orthogonality_mass}, $\vvec_{\ip}^{(3)}$ and $\omega^{(3)}_{\ip}$ correspond to a Cartesian product of Gauss Hermite formula with nodes $\vvec_{\ip}^{(3)} = (v_i,v_j,v_k)$ and $\omega^{(3)}_{\ip} = \omega_i\omega_j\omega_k$ with $\ip=(N+1)^2i+(N+1)j+k$.
\par
The basis of $P^N(\R3)$ are scaled Hermite polynomials:
\begin{equation}
H_{i,j,k}(\vvec) := h_i(\vvec_1)h_j(\vvec_2)h_k(\vvec_3),\quad 0\leq i+j+k\leq N,
\label{eq:Hermite_basis}
\end{equation}
with $h_i$ denoting the scaled 1D Hermite polynomial of degree $i$ \cite{Abramowitz_1964,shen2011spectral}. The scaling is such that $\int e^{-|v|^2}h_j(v)^2\,dv =1$. The polynomial degree of $H_{i,j,k}$ is $i+j+k$, the polynomial degree w.r.t. $(\vvec_1,\vvec_2)$ is $i+j$. The mass matrix in the Hermite basis is also diagonal, i.e.
\begin{equation*}
\intrd e^{-|\vvec|^2}H_{i,j,k}(\vvec)H_{i',j',k'}(\vvec)\,d\vvec = \delta_{i,i'}\delta_{j,j'}\delta_{k,k'}, \quad\quad i+j+k \text{ and }i'+j'+k' = 0\ldots N.
\end{equation*}

For given $\Vbar$ and $\Tbar$, the polynomials are scaled and shifted in the argument to incorporate the above orthogonality relations in the basis. Thus, we have
\begin{equation*}
V_N = \myspan\{H_{i,j,k}\left(\tfrac{\vvec-\Vbar}{\sqrt{\Tbar}}\right) , i+j+k=0\ldots N \},\quad V_{\Tbar,\Vbar,N} = e^{-\bigl|\frac{\vvec-\Vbar}{\sqrt{\Tbar}}\bigr|^2}V_N.
\end{equation*}
With the above notation we additionally emphasize the scaling of the basis functions we use in practice.
\section{Efficient evaluation of the collision integral}
For the collision algorithm we need some additional preparations. First of all, we use the translational invariance of $Q$ to make the collision integrals $\int_{\R3} Q(f)\phi\,d\vvec$ independent of $\Vbar$. We also rescale $\vvec$ to obtain an explicit dependency on $\Tbar$. Let $\frac{\vvec-\Vbar}{\sqrt{\Tbar}}=\tilde{\vvec}$ and $\frac{\wvec-\Vbar}{\sqrt{\Tbar}}=\tilde{\wvec}$ in $\int Q(f)\phi\,d\vvec$, remove the tilde signs to obtain
\begin{equation*}
\intrd Q(f)\phi\,d\vvec = \Tbar^{3+\frac{\beta}{2}}\intrd Q(f^{1,0})\phi^{1,0}\,d\vvec,
\end{equation*}
with centred density function $f^{0,1}(t,\vvec):=f(t,\sqrt{\Tbar}\vvec+\Vbar)$. Thus, our considerations concerning the collision operator can be done without shifted and scaled velocity, i.e. with centred bases.
\par
Next we provide a representation for $\int Q(f)\phi\,d\vvec$ needed for our algorithm.
\begin{lemma}
The collision operator $Q(f)$ defined in \refer{eq:collop} satisfies \cite{Cercignani_1994}
\label{lemma:collrepresentation}
\begin{equation*}
\intrd Q(f)\phi(\vvec)\,d\vvec = \intrd\intrd\int\limits_{S^{2}}B(\vvec,\wvec,\evec')f(\vvec)f(\wvec)[\phi(\vvec')-\phi(\vvec)]d\evec'd\wvec d\vvec
\end{equation*}
\end{lemma}
\begin{remark}
A straight forward evaluation of the collision integrals for $f\in V_{\Tbar,\Vbar,N}$ requires evaluation of
\begin{equation*}
\intrd Q(f)\phi \,d\vvec = \sum\limits_{n=0}^{\ndofv-1}\sum\limits_{m=0}^{\ndofv-1}c_nc_mq_{n,m,j}\quad\quad j = 0\ldots \ndofv-1,
\end{equation*}
where $q_{n,m,j} = \int\int\int B(\vvec,\wvec,\evec') [ L^{(N)}_n(\vvec')L^{(N)}_m(\wvec')-L^{(N)}_n(\vvec)L_m(\wvec)]L^{(N)}_j(\vvec)d\evec'd\vvec d\wvec$.
Evaluating the above double sum for each $j$ requires $\ndofv^3 = (N+1)^9$ operations.
\end{remark}
\subsection{Collision integrals in mean and relative velocity}
Next, we express the collision integral in mean and relative velocities. We start with the representation from lemma \ref{lemma:collrepresentation}. If we let $\vbar:= \frac{\vvec+\wvec}{2},\,\hat{\vvec}:= \frac{\vvec-\wvec}{2}$ we obtain
\begin{equation}
\begin{aligned}
&\intrd Q(f(\vvec))\phi(\vvec)\,d\vvec \\&= 8\intrd \intrd \int\limits_{S^2}b_r(2|\hat{\vvec}|)b_\theta(\tfrac{\hat{\vvec}\cdot \evec'}{|\hat{\vvec}|}) f(\vbar+\hat{\vvec})f(\vbar-\hat{\vvec})[\phi(\vbar+\evec'|\hat{\vvec}|)-\phi(\vbar+\hat{\vvec})]\,d\evec'd\hat{\vvec}\,d\vbar.
\end{aligned}
\end{equation}
Throughout the integrand, $f$ and $\phi$ are evaluated centered at $\vbar$. This is a shift of the coordinate origin. Thus, we approximate $f$ not on the grid associated with the Lagrange polynomials, but on a grid shifted by the mean velocity $\vbar$. We denote this by $f^{\vbar}(\hat{\vvec}):=f(\vbar+\hat{\vvec})$, the shifted test function is $\phi^{\vbar}(\vhat) = \phi(\vbar+\vhat)$. Thus,
\begin{equation}
\begin{aligned}
&\intrd Q(f)(\vvec)\phi(\vvec)\,d\vvec \\&= 8\intrd \intrd \int\limits_{S^2}b_r(2|\hat{\vvec}|)b_\theta(\tfrac{\hat{\vvec}\cdot \evec'}{|\hat{\vvec}|}) f^{\vbar}(\hat{\vvec})f^{\vbar}(-\hat{\vvec})[\phi^{\vbar}(\evec'|\hat{\vvec}|)-\phi^{\vbar}(\hat{\vvec})]\,d\evec'd\hat{\vvec}\,d\vbar.
\end{aligned}
\end{equation}
Finally, letting $f_2^{\vbar}(\hat{\vvec}):=f^{\vbar}(\hat{\vvec})f^{\vbar}(-\hat{\vvec})$, we arrive at
\begin{align}
\intrd Q(f(\vvec))&\phi(\vvec)\,d\vvec = 8\intrd \intrd\int\limits_{S^2} b_r(2|\hat{\vvec}|)b_\theta(\tfrac{\hat{\vvec}\cdot \evec'}{|\hat{\vvec}|})f_2^{\vbar}(\hat{\vvec})[\phi^{\vbar}(\evec'|\hat{\vvec}|)-\phi^{\vbar}(\hat{\vvec})]\,d\evec'd\hat{\vvec} \\ & =\sqrt{2}^3\intrd \underbrace{\intrd\int\limits_{S^2} b_r(|\sqrt{2}\hat{\vvec}|)b_\theta(\tfrac{\hat{\vvec}\cdot \evec'}{|\hat{\vvec}|})f_2^{\vbar}(\tfrac{\hat{\vvec}}{\sqrt{2}})[\phi^{\vbar}(\evec'|\tfrac{\hat{\vvec}}{\sqrt{2}}|)-\phi^{\vbar}(\frac{\hat{\vvec}}{\sqrt{2}})]\,d\evec'd\hat{\vvec}}_{:=Q^I(b_r f^{\vbar}_2,\phi^{\vbar})(\vbar)}\,d\vbar \\
&= \sum\limits_{\ip=0}^{n_\ip-1} \omega_\ip S^{\frac{\vvec_\ip}{\sqrt{2}},\phi} Q^I(b_r f^{\frac{\vvec_\ip}{\sqrt{2}}}_2,\phi^{\frac{\vvec_\ip}{\sqrt{2}}})(\tfrac{\vvec_\ip}{\sqrt{2}}).
\label{eq:inner_coll_op}
\end{align}
The scaling by $\tfrac{1}{\sqrt{2}}$ is to obtain Maxwellians with temperature 1 w.r.t. $\vbar$ and $\vhat$. As before, $(\omega^{(3)}_i,\vvec^{(3)}_i),\,i=0\ldots n_\ip-1$ are weights and nodes of a 3D Gauss-Hermite quadrature rule. We refer to $Q^I$ as the inner collision operator.
\subsubsection{Calculating \texorpdfstring{$f_2^{\vbar}$}{calcf2}}
\label{subsec:f2}
The point wise multiplication $f^{\vbar}(\vhat)f^{\vbar}(-\vhat)=f_2^{\vbar}(\vhat)$ can be done efficient in the Lagrange basis, so we start with $f(\vvec) = e^{-|\vvec|^2}\sum c_jL_j(\vvec)$. From 
\begin{equation*}
f_2^{\vbar}(\hat{\vvec}) = f^{\vbar}(\hat{\vvec})f^{\vbar}(-\hat{\vvec}) = e^{-|\vbar+\vhat|^2}p(\vbar+\vhat)e^{-|\vbar-\vhat|^2}p(\vbar-\vhat),\quad\quad p \in P^{N}(\R3)
\end{equation*}
we deduce that the appropriate Maxwellian for the $f_2^{\vbar}$ is $e^{-2|\vbar|^2 -2|\vhat|^2}$ and that it's polynomial degree is $2N$. Thus, we expand 
\begin{equation*}
f_2^{\vbar}(\vhat)=e^{-2|\vbar|^2-2|\vhat|^2}\sum\limits_{i=0}^{\ndofv^{(2)}-1}e_i L_i^{(2N)}(\vhat),\quad\quad \ndofv^{(2)}:=(2N+1)^3.
\end{equation*}
The Lagrange polynomials collocation properties lead to 
\begin{equation}
e_j = p(\mathbf{\vbar+\mu}_j)p(\vbar-\mathbf{\mu}_j),\quad\quad j=0\ldots \ndofv^{(2)}-1,
\label{eq:f2}
\end{equation}
where $\mathbf{\mu}_j$ are the collocation nodes for $L^{(2N)}$. Incorporating the scaling from \refer{eq:inner_coll_op} we use $\mathbf{\mu}_j = \frac{1}{\sqrt{2}}\vvec_j$, with $\vvec_j$ the Gauss Hermite nodes of order $2N$. 
The values $p(\vbar+\mu_j)$ are obtained from
\begin{equation*}
f(\vbar+\vhat)=e^{-|\vbar+\vhat|^2}\sum\limits_{i=0}^{\ndofv-1} c_i L_i^{(N)}(\vbar+\vhat) = e^{-|\vbar+\vhat|^2} p(\vbar+\vhat). 
\end{equation*}
Using $\vhat = \mu_j$ results in
\begin{equation}
p(\vbar+\mu_j) = \sum\limits_{i=0}^{\ndofv-1} c_i L_i^{(N)}(\vbar+\mu_j),\quad\;\; j = 0\ldots \ndofv^{(2)}-1.
\label{eq:evaluate_shifted}
\end{equation}
Due to symmetry of the nodes, $p(\vbar-\mu_j) = p(\vbar + \mu_{\ndofv^{(2)}-1-j})$.
A factorization into one dimensional shifts can be applied to the sum in \refer{eq:evaluate_shifted}, resulting in $\mathcal{O}(N^4)$ complexity for $p(\vbar+\mu_j),j=0\ldots \ndofv^{(2)}-1$. We denote the shift matrix for a given $\vbar$ by $S^{\vbar}$, $S^{\vbar}_{ji} = L_i(\vbar+\mu_j)$.
\par 
The calculation of the test functions $\Lvec_j(\vbar+\vhat)$ from the shifted Lagrange polynomials $\Lvec^{\vbar}(\vhat)$ is achieved by the transposed matrix ${S^{\vbar}}^t$, i.e. $\Lvec(\vbar+\vhat) = {S^{\vbar}}^t \Lvec^{\vbar}(\vhat)$, see \cite{Kitzler2015,Kitzler2016} for more details.

\subsection{The \texorpdfstring{$\vhat$}{vhatint} integral}
For the inner collision operator $Q^I$ we transform $f^{\tfrac{\vbar}{\sqrt{2}}}_2(\tfrac{\vhat}{\sqrt{2}})$ to a polynomial basis which gives a diagonal $Q^I$. We first introduce the involved polynomial bases and then show how to obtain efficient transformations. We use the hierarchical Hermite polynomial basis from \refer{eq:Hermite_basis}, a hierarchical basis in cylinder and finally a hierarchical basis in spherical coordinates. 

\subsubsection{Cylinder Hermite basis}
We extend the polar basis from \cite{Kitzler2015,Kitzler2016} by a component in $z$ direction. To that end we use the shorthand notation $I(i,j):=2j+i\mod 2$. We define the 2D Polar Laguerre basis of order $i$:
\begin{definition}
The Polar-Laguerre polynomials $\{\Psi_{i,j}^r,\, j=0\ldots \lfloor 0.5 i\rfloor$,\\ $r\in\{\cos,\sin\}\}$ are defined via
\begin{align*}
\Psi_{i,j}^{\cos}((\vvec_1,\vvec_2)) &:= \gamma_{i,j}\cos(I(i,j)\phi)r^{I(i,j)}\mathcal{L}_{\frac{i-I(i,j)}{2}}^{I(i,j)}(r^2) \\
\Psi_{i,j}^{\sin}((\vvec_1,\vvec_2)) &:= \gamma_{i,j}\sin(I(i,j)\phi)r^{I(i,j)}\mathcal{L}_{\frac{i-I(i,j)}{2}}^{I(i,j)}(r^2).
\end{align*}
$(r,\phi)$ are the polar coordinates of $(\vvec_1,\vvec_2)$, The normalization is $\gamma_{0,2j} = \sqrt{\tfrac{2}{\pi}}$ and $\gamma_{i,j} = \sqrt{\tfrac{1}{\pi}}, i\neq 0 \vee j\not\in 2\mathbb{N}$. $\mathcal{L}_k^\alpha(x)$ is the (scaled) associated Laguerre polynomial satisfying $\int_{\R{+}}e^{-x}x^\alpha \mathcal{L}_i^\alpha(x)\mathcal{L}_j^\alpha\,dx =\delta_{i,j}$ \cite{Abramowitz_1964, shen2011spectral}.
\label{def:polarlaguerre}
\end{definition}
Now we construct a basis of $V_N$ in cylinder coordinates.
\begin{definition}
\label{def:cylinderlaguerre}
The Cylinder Hermite polynomials are defined as
\begin{align*}
\PH_{k,i,j}^{\cos}(\vvec) &:= \Psi_{i,j}^{\cos}((\vvec_1,\vvec_2)) h_{k-i}(\vvec_3) \\
\PH_{k,i,j}^{\sin}(\vvec) &:= \Psi_{i,j}^{\sin}((\vvec_1,\vvec_2))h_{k-i}(\vvec_3),
\end{align*}
with $(r,\phi,\vvec_3)$ being the cylinder coordinates of $\vvec$. The indices are $j = 0\dots \lfloor 0.5 i\rfloor$, $i = 0\dots k$ and $k=0\dots N$.
\end{definition}
\begin{remark}
Note that in definitions \ref{def:polarlaguerre} and \ref{def:cylinderlaguerre}, the $\sin$ basis functions only appear when $I(i,j) \neq 0$. We keep that in mind, but for simplicity we do not explicitly distinguish this special case in the notation.
\end{remark}
As is shown in \cite{Kitzler2015,Kitzler2016}, $\Psi^t_{i,j}, t\in \{\cos,\sin\}$ is a polynomial of total degree $i$ w.r.t. $\vvec_1$ and $\vvec_2$. Multiplied by $h_{k-i}(\vvec_3)$ gives a polynomial of total degree $k$ w.r.t. to $\vvec$. \\In the subsequent lemmata we show that the Cylinder Hermite polynomials form a basis of $P^N(\R3)$ and are orthogonal w.r.t. to the Maxwellian weighted $L_2-$inner product.
\begin{lemma}
The Cylinder Hermite polynomials are orthonormal w.r.t. to the Maxwellian weighted $L_2-$inner product, i.e.
\begin{equation*}
\intrd e^{-|\vvec|^2}\PH_{k,i,j}^t(\vvec){\PH}^{t'}_{k',i',j'}(\vvec)\,d\vvec = \delta_{k,k'}\delta_{i,i'}\delta_{j,j'}\delta_{t,t'},\quad r\in\{\cos,\sin\}.
\end{equation*}
\end{lemma}
\begin{proof}
By the orthogonality of the Hermite polynomials we write
\begin{align*}
&\intrd e^{-|\vvec|^2}\PH_{k,i,j}^t(\vvec){\PH}^{t'}_{k',i',j'}(\vvec)\,d\vvec \\
&= \delta_{k-i,k'-i'} \int\limits_{\R2}e^{-\vvec_1^2-\vvec_2^2}\Psi_{i,j}^t((\vvec_1,\vvec_2))\Psi_{i',j'}^{t'}((\vvec_1,\vvec_2) d(\vvec_1,\vvec_2).
\end{align*}
The value of the remaining integral is $\delta_{i,i'}\delta_{j,j'}\delta_{t,t'}$ as is shown in \cite{Kitzler2015,Kitzler2016} what concludes the proof.
\end{proof}
\begin{lemma}
The Cylinder Hermite polynomials $\PH_{k,i,j}^{t}(\vvec),\,k = 0\ldots N,i=0\ldots k$, $j=0\lfloor0.5i\rfloor$ and $t\in\{\sin,\cos\}$ from definition \ref{def:cylinderlaguerre} form a basis of $P^N(\R3)$.
\end{lemma}
\begin{proof}
Since the Cylinder Hermite polynomials are orthogonal, they are also linearly independent. Now for fixed $i$ there are $i+1$ basis polynomials. Summing  over all $i$ and $k$ gives $\tfrac{1}{6}(N+1)(N+2)(N+3)$ linear independent polynomials.
\end{proof}
\subsubsection{Spherical Laguerre Basis}
Let $Y_i^{j,t}(\theta,\varphi), |j|\leq i, (\theta,\varphi) \in [0,\pi]\times[0,2\pi]$,\\$t \in \{\cos,\sin\}$ denote the real valued Spherical Harmonic of degree $i$ and order $j$ \cite{gumerov2004}. These are defined as
$$
Y_i^{j,t}(\theta,\varphi) = c_{i,j}t(j\varphi)P_i^{|j|}(\cos(\theta)),\quad c_{i,j} = s_j \sqrt{\frac{2i+1}{4\pi} \frac{(i-|j|)!}{(i+|j|)!} },
$$
with $s_0 = 1$ and $s_j = \sqrt{2}, j\neq 0$. $P_i^j$ is the Legendre function of degree $i$ and order $j$ \cite{gumerov2004}.
From these we construct a 3D basis in spherical coordinates as given below.
\begin{definition}
The Spherical Laguerre polynomials $\Phi^{t}_{k,i,j}$ are defined via
\begin{align}
\Phi^{\cos}_{2k,i,j}(r,\theta,\varphi) &= \sqrt{2}Y_{2i}^{j,\cos}(\theta,\varphi) r^{2i}\mathcal{L}_{k-i}^{2i+0.5}(r^2),\quad & i\leq k, 0\leq j\leq 2i \\
\Phi^{\sin}_{2k,i,j}(r,\theta,\varphi) &= \sqrt{2}Y_{2i}^{j,\sin}(\theta,\varphi) r^{2i}\mathcal{L}_{k-i}^{2i+0.5}(r^2),\quad & i\leq k, 1\leq j\leq 2i \\
\Phi^{\cos}_{2k+1,i,j}(r,\theta,\varphi) &= \sqrt{2}Y_{2i+1}^{j,\cos}(\theta,\varphi) r^{2i+1}\mathcal{L}_{k-i}^{2i+1.5}(r^2),\quad & i\leq k, 0\leq j\leq 2i+1 \\
\Phi^{\sin}_{2k+1,i,j}(r,\theta,\varphi) &= \sqrt{2}Y_{2i+1}^{j,\sin}(\theta,\varphi) r^{2i+1}\mathcal{L}_{k-i}^{2i+1.5}(r^2),\quad & i\leq k, 1\leq j\leq 2i+1.
\end{align}
$\mathcal{L}_i^j$ is the (scaled) generalized Laguerre polynomial satisfying $\int_{\R{+}}e^{-x}x^\alpha \mathcal{L}_i^\alpha(x)\mathcal{L}_j^\alpha(x) = \delta_{i,j}$ \cite{Abramowitz_1964,shen2011spectral}.
\end{definition}
\begin{lemma}
The Spherical Laguerre polynomials are orthonormal on $\R3$ w.r.t. the Maxwellian weighted inner product, i.e. 
$$
\int\limits_{\R3}e^{-|\vvec|^2} \Phi_{k,i,j}^t\Phi_{k',i',j'}^{t'}\,dv = \delta_{k,k'}\delta_{i,i'}\delta_{j,j'}\delta_{t,t'}.
$$
\label{lemma:orthogonality_sl}
\end{lemma}
\begin{proof}
We start considering the case of even $k$ and $k'$. For simplicity we write $2k$ and $2k'$.
Using spherical coordinates, $\vvec = r\evec$ gives
\begin{align*}
\int\limits_{\R3}\!\!e^{-|\vvec|^2} &\Phi^t_{2k,i,j}\Phi^{t'}_{2k',i',j'}\,d\vvec \\&= 2\!\!\int\limits_0^\infty\int\limits_{S^2} \!\!e^{-r^2} r^{2i+2i'+2}\mathcal{L}_{k-i}^{2i+0.5}(r^2)\mathcal{L}_{k'-i'}^{2i'+0.5}(r^2)Y_{2i}^{j,t}(\evec)Y_{2i'}^{j',t'}\!\!(\evec)d\evec dr \\
&=2\delta_{i,i'}\delta_{j,j'}\delta_{t,t'} \int\limits_0^\infty e^{-r^2}r^2 r^{4i}\mathcal{L}_{k-i}^{2i+0.5}(r^2)\mathcal{L}_{k'-i}^{2i+0.5}(r^2)dr \\
& \stackrel{r^2=\tilde{r}}{=}\delta_{i,i'}\delta_{j,j'}\delta_{t,t'}\int\limits_0^\infty e^{-\tilde{r}}\tilde{r}^{2i+0.5}\mathcal{L}_{k-i}^{2i+0.5}(\tilde{r})\mathcal{L}_{k'-i}^{2i+0.5}(\tilde{r})d\tilde{r} \\
& = \delta_{k,k'}\delta_{i,i'}\delta_{j,j'}\delta_{t,t'}.
\end{align*}
The deltas w.r.t. $i$, $j$ and $t$ are due to the Spherical Harmonics forming an orthonormal basis on $S^2$. The delta in $k$ is due to the orthogonality of the generalized Laguerre polynomials $\mathcal{L}_k^\alpha$. For $k$ and $k'$ both odd, the proof is the same. For even $k$ only $Y_{2i}^{j,t}$ arise in $\Phi_{k,i,j}^t$ and $Y_{2i+1}^{j,t}$ arise for odd $k'$. Thus, such combinations -- and vice versa -- yield 0 contribution.
\end{proof}
\begin{lemma}
The Spherical Laguerre polynomials $\Phi^t_{k,i,j},\, k\leq N$ form a basis of $P^N(\R3)$.
\end{lemma}
\begin{proof}
Since $Y^{j,t}_{2i}$ and $Y^{j,t}_{2i+1}$ are polynomials of total degree $2i$ and $2i+1$, multiplication with $\mathcal{L}^{2i+0.5}_{k-i}$ and $\mathcal{L}^{2i+1.5}_{k-i}(r^2)$ gives a polynomial of total degree $2k$ and $2k+1$ w.r.t. $\vvec$. So each $\Phi^t_{k,i,j} \in P^N(\R3)$.
The linear independence follows from lemma \ref{lemma:orthogonality_sl}. For a fixed $k$, there are $\tfrac{(k+1)(k+2)}{2}$ polynomials. 
Summing $k$ from 0 to $N$ yields $\tfrac{1}{6}(N+1)(N+2)(N+3)$ linear independent polynomials.
\end{proof}
Similar to the proof of lemma \ref{lemma:orthogonality_sl} one obtains 
\begin{lemma}
The Spherical Laguerre polynomials are orthogonal w.r.t. the \\Maxwellian weighted inner collision operator, i.e.
\begin{equation*}
Q^I(e^{-|\vhat|^2}\Phi^t_{k,i,j},\Phi^{t'}_{k',i',j'}) = 
d_{k,i,j}\delta_{k,k'}\delta_{i,i'}\delta_{j,j'}.
\end{equation*}
The constant $d_{k,i,j}$ is
\begin{equation*}
d_{k,i,j} = \begin{cases}
2\pi\int\limits_{-1}^1 b_\theta(\mu)(P_{2i}(\mu)-1)d\mu & k\in 2\mathbb{N} \\ 
2\pi\int\limits_{-1}^1 b_\theta(\mu)(P_{2i+1}(\mu)-1)d\mu
& k\in 2\mathbb{N}+1\end{cases}.
\end{equation*}
\end{lemma}
\begin{proof}
We consider again the case of even $k$ and $k'$ and again write $2k$ and $2k'$ instead of $k$ and $k'$. As in the proof of lemma \ref{lemma:orthogonality_sl}, $\vhat$ is transformed to spherical coordinates.
\begin{align*}
Q^I(e^{-|\vhat|^2}\Phi^t_{k,i,j},\Phi^{t'}_{k',i',j'}) = &\;2\!\!\int\limits_0^\infty\!\! e^{-r^2}r^{2+2i+2i'}\mathcal{L}_{k-i}^{2i+0.5}(r^2)\mathcal{L}_{k'-i'}^{2i'+0.5}(r^2)dr \\\times &\int\limits_{S^2}\int\limits_{S^2}b_\theta(\evec\cdot \evec')Y_{2i}^{j,t}(\evec)(Y_{2i'}^{j',t'}(\evec')-Y_{2i'}^{j',t'}(\evec))d\evec'd\evec.
\end{align*}
The surface integral can be rewritten as
\begin{align*}
\int\limits_{S^2}&\int\limits_{S^2}b_\theta(\evec\cdot \evec')Y_{2i}^{j,t}(\evec)(Y_{2i'}^{j',t'}(\evec')-Y_{2i'}^{j',t'}(\evec))d\evec'd\evec \\ &= \underbrace{\int\limits_{S^2}\!\!\left(\;\int\limits_{S^2} b_\theta(\evec\cdot \evec')Y_{2i'}^{j',t'}\!(\evec')d\evec'\right)\!\!Y_{2i}^{j,t}(\evec)d\evec}_{:=A} - \underbrace{\int\limits_{S^2}\!\!\left(\;\int\limits_{S^2} b_\theta(\evec\cdot \evec')d\evec'\right)\!\!Y_{2i'}^{j',t'}\!(\evec)Y_{2i}^{j,t}(\evec)d\evec}_{:=B}.
\end{align*}
Using the Funk-Hecke theorem \cite{gumerov2004} we
evaluate $A$ to
\begin{equation*}
A = \lambda_{2i'}\int\limits_{S^2}Y_{2i'}^{j',t'}(\evec)Y_{2i}^{j,t}(\evec)d\evec = \lambda_{2i'}\delta_{i,i'}\delta_{j,j'}\delta_{t,t'},\; \text{ with } \; \lambda_{2i'} = 2\pi \int\limits_{-1}^1b_\theta(\mu)P_{2i'}(\mu)d\mu.
\end{equation*}
For $B$ we obtain
\begin{equation*}
B = \lambda_0\delta_{i,i'}\delta_{j,j'}\delta_{t,t'}.
\end{equation*}
Combining the values for $A$ and $B$ and using as before the orthogonality of the Laguerre polynomials completes the proof.
\end{proof}
\subsection{Transforming to a sparse inner collision operator}
\label{sec:transformations}
In the present subsection we write $f$ instead of $f^{\nicefrac{\vbar}{\sqrt{2}}}_2$ to simplify notation. We assume that this function is given as an expansion to Lagrange polynomials $L^{(N)}_j$ with collocation nodes $\frac{1}{\sqrt{2}}\vvec_j$, where $\vvec_j$ is the $j-$th Gauss Hermite node.
\subsubsection{Nodal to Hermite representation}
\label{sec:trafo_n_to_h}
We use $\sum\limits_{l,m,n}$ to denote $\sum\limits_{l=0}^{2N}\sum\limits_{m=0}^{2N}\sum\limits_{n=0}^{2N}$, as well as $\sum\limits_{i,j,k}$ to denote $\sum\limits_{i=0}^N\sum\limits_{j=0}^{N-i}\sum\limits_{k=0}^{N-i-j}$.
Transforming nodal to hierarchical requires

\begin{equation*}
f(\tfrac{\vvec}{\sqrt{2}}) = e^{-|\vvec|^2}\!\! \sum\limits_{l,m,n}c_{l,m,n}l_l(\frac{\vvec_1}{\sqrt{2}})l_m(\frac{\vvec_2}{\sqrt{2}})l_n(\frac{\vvec_3}{\sqrt{2}}) = e^{-|\vvec|^2}\sum\limits_{i,j,k}h_{i,j,k}h_i(\vvec_1)h_j(\vvec_2)h_k(\vvec_3),
\end{equation*}
where the equal sign has to be understood in terms of an $L_2$ projection.
Projecting onto $\myspan\{H_{i,j,k}:i+j+k\leq N\}$ gives
\begin{equation*}
h_{i',j',k'} = \sum\limits_{l,m,n} c_{l,m,n}\intrd e^{-|v|^2}h_{i'}(\vvec_1)h_{j'}(\vvec_2)h_{k'}(\vvec_3)l_l(\frac{\vvec_1}{\sqrt{2}})l_m(\frac{\vvec_2}{\sqrt{2}})l_n(\frac{\vvec_3}{\sqrt{2}})\,d\vvec.
\end{equation*}
We define the 1D projection matrix $\mathcal{P}^{L\rightarrow H} \in \R{(N+1)\times (2N+1)}$, with \\$\mathcal{P}^{L\rightarrow H}_{i,j} := \int e^{-v^2}h_i(v)l_j(\frac{v}{\sqrt{2}})dv$ and obtain
\begin{equation*}
h_{i',j',k'} = \sum\limits_{l,m,n} c_{l,m,n}\mathcal{P}^{L\rightarrow H}_{i',l}
\mathcal{P}^{L\rightarrow H}_{j',m}
\mathcal{P}^{L\rightarrow H}_{k',n} = 
\sum\limits_{l} \mathcal{P}^{L\rightarrow H}_{i',l}
\underbrace{\sum\limits_{m}\mathcal{P}^{L\rightarrow H}_{j',m}
\underbrace{\sum\limits_{n}c_{l,m,n} \mathcal{P}^{L\rightarrow H}_{k',n}}_{:=h^1_{l,m,k'}}}_{:=h^2_{l,j',k'}}.
\end{equation*}
We evaluate the triple sum via 3 single sums, corresponding to 3 1D transformations. Each of them requires $N^4$ operations for evaluation, giving $\mathcal{O}(N^4)$ complexity.
Additionally, these transforms are executed as matrix multiplications, using efficient Lapack routines \cite{laug}.
\par
The entries of the transformation matrix evaluate to
\begin{equation*}
\mathcal{P}^{L\rightarrow H}_{i,j}=\int e^{-v^2}h_i(v)l_j(\frac{v}{\sqrt{2}})dv = \sum\limits_\text{ip} \omega_{\text{ip}}h_i(v_\text{ip})l_j(\frac{v_\text{ip}}{\sqrt{2}}) = \omega_j h_i(v_j).
\end{equation*}
\subsubsection{Hermite to Cylinder Hermite representation}
In this subsection we use $\sum\limits_{k,i,j,t}$ to denote $\sum\limits_{k=0}^N\sum\limits_{i=0}^k\sum\limits_{j=0}^{\lfloor 0.5 i\rfloor}\sum\limits_{t \in \{\cos,\sin\}}$. Note that we do not overload notation by explicitly deducing the cases ($j=0 \wedge i\in 2\mathbb{N}$), where no $\sin$ basis polynomial exists.
Transforming to Cylinder Hermite polynomials requires 
\begin{equation*}
f(\tfrac{\vvec}{\sqrt{2}}) = e^{-|\vvec|^2}\sum\limits_{i+j+k\leq N}h_{i,j,k}h_i(\vvec_1)h_j(\vvec_2)h_k(\vvec_3) = e^{-|\vvec|^2}\sum\limits_{k,i,j,r}\ph_{k,i,j,t}\PH_{k,i,j}^t(\vvec).
\end{equation*}
Projecting the Hermite polynomial representation onto the Cylinder Hermite polynomials results in
\begin{equation*}
\sum\limits_{i+j+k\leq N}h_{i,j,k}\intrd e^{-|\vvec|^2}H_{i,j,k}(\vvec)\PH_{k',i',j'}^{t'}(\vvec) \,d\vvec  = \ph_{k',i',j',t'}.
\end{equation*}
On the left side we obtain:
\begin{align*}
\intrd e^{-|\vvec|^2}H_{i,j,k}(\vvec)\PH_{k',i',j'}^{t'}(\vvec) \,d\vvec = &
\overbrace{\int\limits_{\R{}}e^{-\vvec_3^2}h_k(\vvec_3)h_{k'-i'}(\vvec_3)d\vvec_3}^{=\delta_{k,k'-i'}} \times \\ &\underbrace{\int\limits_{\R2} e^{-\vvec_1^2-\vvec_2^2}h_i(\vvec_1)h_j(\vvec_2)\Psi_{i',j'}^{t'}((\vvec_1,\vvec_2)) \,d(\vvec_1,\vvec_2)}_{:=\delta_{i+j,i'} p^{H\to\PH}_{i,j,i',j',t'}}
\end{align*}
The two deltas can be written as $\delta_{k,k'-i'}\delta_{i+j,i'} = \delta_{i+j+k,k'}\delta_{i+j,i'}$. Consequently, only equal degree polynomials ($i+j+k$ and $k'$) interact. The second delta $\delta_{i+j,i'}$ is discussed in \cite{Kitzler2015,Kitzler2016}. Again, only equal degree polynomials w.r.t. $(\vvec_1,\vvec_2)$ interact ($i+j$ and $i'$). Using the deltas we can write
\begin{equation*}
\sum\limits_{i=0}^{i'} h_{i,i'-i,k'-i'} \mathcal{P}^{H\to\PH}_{i,i'-i,i',j',t'} = \ph_{k',i',j',t'}, \quad j'=0\dots\lfloor 0.5 i\rfloor, i'=0\ldots k',k'=0\ldots N.
\end{equation*}
This gives a very structured transformation from Hermite to Cylinder Hermite polynomials: We consider a fixed $0\leq k'\leq N$ and $0\leq i'\leq k'$ and denote
$\mathbf{h}_{k',i'} := (h_{0,i',k'-i'},h_{1,i'-1,k'-i'},\ldots h_{i',0,k'-i'})^t$ and $\mathbf{\ph}_{k',i'}:=(\mathbf{\ph}_{k',i',\cos},\mathbf{\ph}_{k',i',\sin})^t$ with 
$\mathbf{\ph}_{k',i',t'} := (\ph_{k',i',0,t'},\ph_{k',i',1,t'},\ldots \ph_{k',i',\lfloor 0.5 i\rfloor,t'})^t$. Let us now define the projection matrix $\mathcal{P}^{H\to\PH}_{i'}\in \R{(i'+1)\times (i'+1)}$ with
\begin{equation*}
\mathcal{P}^{H\to\PH}_{i'}:=\left(\begin{matrix}\mathcal{P}^{H\to\PH}_{i',\cos} \\ \mathcal{P}^{H\to\PH}_{i',\sin}\end{matrix}\right), \quad\quad\!\!
\mathcal{P}^{H\to\PH}_{i',t'} := \left(
\begin{matrix}
p^{H\to\PH}_{0,i',i',0,t'} &\ldots & p^{H\to\PH}_{i',0,i',0,t'} \\
p^{H\to\PH}_{0,i',i',1,t'} &\ldots & p^{H\to\PH}_{i',0,i',1,t'} \\
\vdots & \vdots & \vdots \\
p^{H\to\PH}_{0,i',i',\lfloor 0.5 i\rfloor ,t'}  &\ldots & p^{H\to\PH}_{i',0,i',\lfloor 0.5 i\rfloor,t'}
\end{matrix}
\right)
\end{equation*}
to write the transformation as $\mathbf{\ph}_{k',i'} = \mathcal{P}^{H\to\PH}_{i'} \mathbf{h}_{k',i'},\, 0\leq i'\leq k',0\leq k'\leq N$.
The matrix $\mathcal{P}^{H\to\PH}_{i'}$ is independent of $k'$ such that it can be recycled for different total polynomial degrees $k'$. Thus, we arrange the Hermite and Cylinder Hermite coefficients as 
\begin{equation*}
\mathbf{H} =
\left(\begin{matrix}
\mathbf{h}_{0,0} & \mathbf{h}_{1,0} & \mathbf{h}_{2,0} &\ldots &\mathbf{h}_{N,0} \\
0                & \mathbf{h}_{1,1} & \mathbf{h}_{2,1} &\ldots &\mathbf{h}_{N,1} \\
0                & 0                & \mathbf{h}_{2,2} &\ldots &\mathbf{h}_{N,2} \\
\vdots			 & \vdots			& \vdots 		   & \vdots& \vdots \\
0                &                  & \ldots           & 0     &\mathbf{h}_{N,N} \\
\end{matrix}\right)\quad\!\!
\mathbf{\Theta} =
\left(\begin{matrix}
\mathbf{\ph}_{0,0} & \mathbf{\ph}_{1,0} & \mathbf{\ph}_{2,0} &\ldots &\mathbf{\ph}_{N,0} \\
0                & \mathbf{\ph}_{1,1} & \mathbf{\ph}_{2,1} &\ldots &\mathbf{\ph}_{N,1} \\
0                & 0                & \mathbf{\ph}_{2,2} &\ldots &\mathbf{\ph}_{N,2} \\
\vdots			 & \vdots			& \vdots 		   & \vdots& \vdots \\
0                &                  & \ldots           & 0     &\mathbf{\ph}_{N,N} \\
\end{matrix}\right),
\end{equation*}
and define the block diagonal matrix $\mathcal{P}^{H\to\PH}:=\mathrm{diag}(\mathcal{P}^{H\to\PH}_0,\ldots \mathcal{P}^{H\to\PH}_N)$ (see figure \ref{fig:block_diagonal_trafo_h_ph}) to obtain
\begin{equation*}
\mathbf{\Theta}=\mathcal{P}^{H\to\PH}\mathbf{H}.
\end{equation*}
The required operations for the transformation for a fixed $(k',i')$ is $(i'+1)^2$. Thus, summing over $i'=0\ldots k'$ and $k' = 0\ldots N$ results in $\mathcal{O}(N^4)$ operations. The arising matrix multiplications are again performed with highly optimized Lapack routines \cite{laug}.
\begin{figure}
\centering
\includegraphics[scale=0.65]{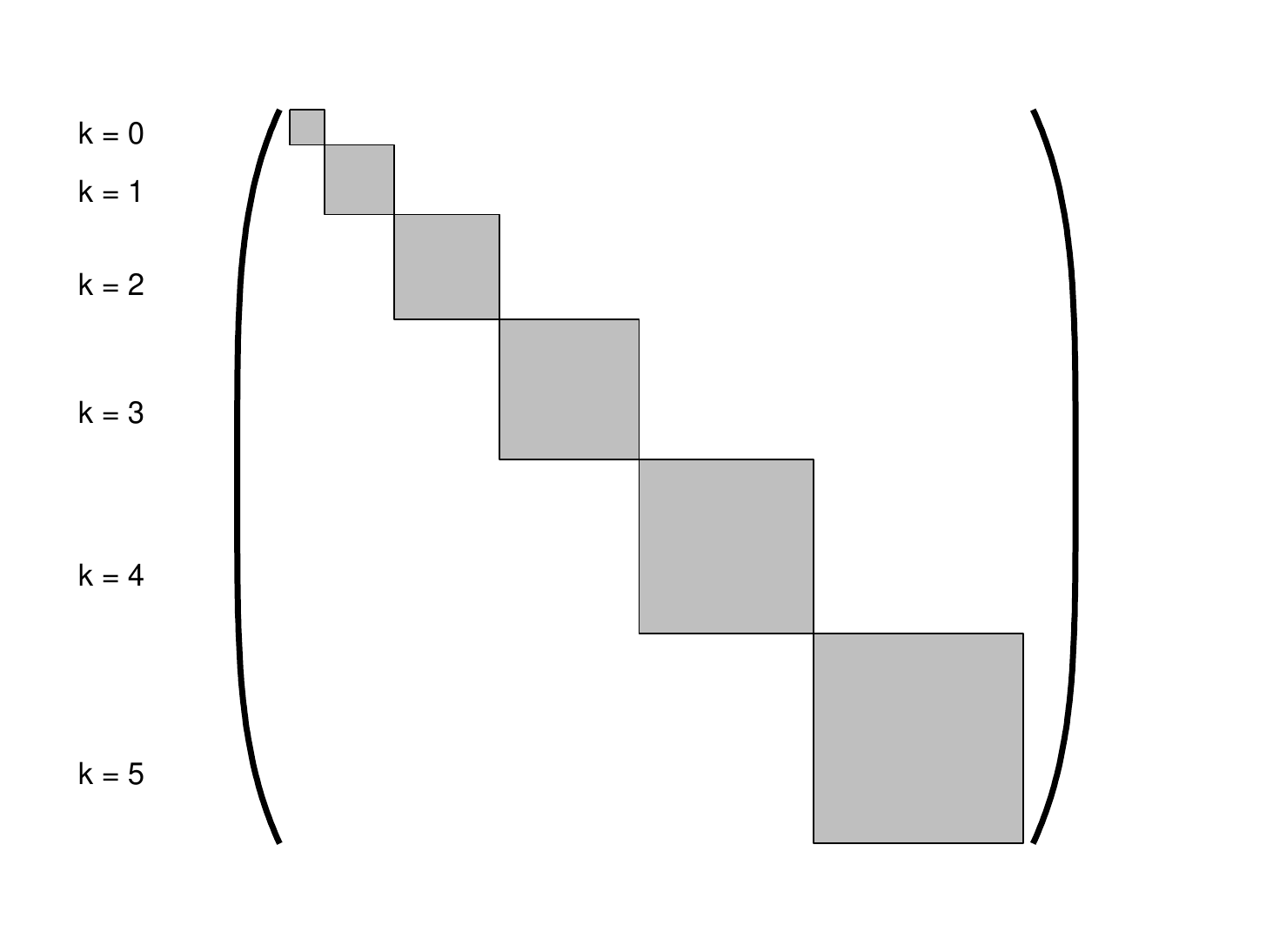}
\caption{The structure of the projection matrix $\mathcal{P}^{H\to\PH}$. The gray shaded blocks are the only non zero entries, the $i'$-th block is of size $(i'+1) \times (i'+1)$.}
\label{fig:block_diagonal_trafo_h_ph}
\end{figure}
\subsubsection{Cylinder Hermite to Spherical Laguerre representation}
As before we use the shorthand notation $\sum\limits_{k,i,j,t}$ to denote 
\begin{equation*}
\sum\limits_{k=0}^N\sum\limits_{i=0}^k\sum\limits_{j=0}^{\lfloor 0.5 i\rfloor}\sum\limits_{t \in \{\cos,\sin\}},
\end{equation*}
the sum over the Cylinder Hermite polynomials. Additionally we use $\!\!\sum\limits_{k',i',j',t'}\!\!\!$ instead of 
\begin{equation*}
\sum\limits_{k'=0}^N\sum\limits_{i'=0}^{k'}\sum\limits_{j'=0}^{2 i' + I_{2\mathbb{N}}(k')}\sum\limits_{t' \in \{\cos,\sin\}}
\end{equation*}
for the sum over the Spherical Laguerre polynomials. $I_{2\mathbb{N} }(k')$ is the indicator function of the set $2\mathbb{N}$ at $k'$. We require
\begin{equation*}
f(\tfrac{\vvec}{\sqrt{2}}) = e^{-|v|^2}\sum\limits_{k,i,j,t}\ph_{k,i,j}^t\PH_{k,i,j}^t(\vvec) = e^{-|v|^2}\sum\limits_{k',i',j',t'} \phi_{k',i',j'}^{t'} \Phi_{k',i',j'}^{t'}(\vvec).
\end{equation*}
Projecting onto the Spherical Laguerre polynomials results in 
\begin{equation*}
\sum\limits_{k,i,j,t}\ph_{k,i,j}^t\underbrace{\intrd e^{-|v|^2}\PH_{k,i,j}^t(\vvec)\Phi_{k',i',j'}^{t'}(\vvec)\,d\vvec}_{:=\mathcal{P}^{\PH\to \Phi}_{(k',i'),(i,j)}} = \phi_{k',i',j'}^{t'}.
\end{equation*}
Since the Cylinder Hermite and the Spherical Laguerre bases are orthogonal w.r.t. the same inner product and are hierarchical, the only interaction is when $k' = k$. Accordingly we can simplify the left hand side to
\begin{equation}
\sum\limits_{i,j,t}\ph_{k',i,j}^t\intrd e^{-|v|^2}\PH_{k',i,j}^t(\vvec)\Phi_{k',i',j'}^{t'}(\vvec)\,d\vvec = \phi_{k',i',j'}^{t'}.
\label{eq:sum_cylinderhermite}
\end{equation}
Spherical coordinates for $\vvec$ and $x^2+y^2 = r^2\sin^2(\theta)$, as well as the abbreviation $2i+(k\mod 2)=:K(i,k)$ transform the integral in \refer{eq:sum_cylinderhermite} into
\begin{align*}
\mathcal{P}^{\PH\to \Phi}_{(k',i'),(i,j)}=\,&\delta_{j',I(i,j)}\delta_{t,t'} \int\limits_{\R+}\int\limits_{[-\pi,\pi]}e^{-|r|^2}(r\sin(\theta))^{I(i,j)}\mathcal{L}_{\frac{i-I(i,j)}{2}}^{I(i,j)}(r^2\sin^2(\theta))\times \\ & h_{k'-i}(r\cos(\theta)) P^{j'}_{K(i',k')}(\cos(\theta)) r^{K(i',k')}\mathcal{L}_{k'-i'}^{K(i',k')+0.5}(r^2)r^2\sin(\theta) \,d\theta dr.
\end{align*}
The deltas in the above integral simplify the sum in \refer{eq:sum_cylinderhermite} to
\begin{equation*}
\phi_{k',i',j'}^{t'} = \sum\limits_i\ph_{k',i,\tfrac{j'-(i\!\!\!\!\mod 2)}{2}}^{t'} \mathcal{P}^{\PH\to \Phi}_{(k',i'),(i,\tfrac{j'-(i\!\!\!\!\mod 2)}{2})}.
\end{equation*}
We need to  specify the range for $i$, $i'$ and $j'$. From $\delta_{j',I(i,j)}$ we deduce that for non vanishing integrals either $j'\in 2\mathbb{N} \wedge i\in 2\mathbb{N}$ or $j'\in 2\mathbb{N}+1 \wedge i\in 2\mathbb{N}+1$.
Using $I(i,j) = j'$ we obtain the bounds for $i$:
\begin{equation*}
i=0\ldots k' \wedge j=0\ldots \lfloor 0.5i\rfloor \Leftrightarrow j=0\ldots \lfloor 0.5k'\rfloor \wedge i=j'\ldots k'.
\end{equation*}
The bounds for the $j'$ and $i'$ result in 
\begin{equation*}
i'=0\ldots \lfloor \frac{k'}{2}\rfloor \wedge j'=0\ldots K(i',k') \Leftrightarrow j'=0\ldots k' \wedge i'= \lfloor \frac{j'-k\!\!\!\!\mod 2+1}{2}\rfloor\ldots \lfloor \frac{k'}{2} \rfloor.
\end{equation*}
With these two equations, we can state the transformation, which -- for $k'\in 2\mathbb{N}$ -- is 
\begin{align*}
\phi^{\cos}_{k',i',0} & = \sum\limits_{i=0}^{\frac{k'}{2}}\ph_{k',2i,0}^{\cos}\mathcal{P}^{\PH\to \Phi}_{(k',i'),(2i,0)}, \quad i' = 0\ldots \frac{k'}{2} \\
\text{for } j' = 1\ldots & \frac{k'}{2}:\\
\phi^{t'}_{k',i',2j'-1} & = \sum\limits_{i=j'}^{\frac{k'}{2}}\ph_{k',2i-1,j'-1}^{t'}\mathcal{P}^{\PH\to \Phi}_{(k',i'),(2i-1,j'-1)}, \quad t'\in\{\cos,\sin\},\;i' = j'\ldots \frac{k'}{2} \\
\phi^{t'}_{k',i',2j'} & = \sum\limits_{i=j'}^{\frac{k'}{2}}\ph_{k',2i,j'}^{t'}\mathcal{P}^{\PH\to \Phi}_{(k',i'),(2i,j')}, \quad t'\in\{\cos,\sin\},\;i' = j'\ldots \frac{k'}{2}.
\end{align*}
This is a block diagonal matrix, with the first block of size $(\nicefrac{k'}{2}+1)\times (\nicefrac{k'}{2}+1)$ followed by 4 blocks of size $(\nicefrac{k'}{2}-j+1)\times (\nicefrac{k'}{2}-j+1), j=1\ldots \nicefrac{k'}{2}$.
For odd $k'$ we arrive at
\begin{align*}
\phi^{\cos}_{k',i',0} & = \sum\limits_{i=0}^{\frac{k'-1}{2}}\ph_{k',2i,0}^{\cos}\mathcal{P}^{\PH\to \Phi}_{(k',i'),(2i,0)}, \quad i' = 0\ldots \frac{k'-1}{2} \\
\phi^{t'}_{k',i',1} & = \sum\limits_{i=0}^{\frac{k'-1}{2}}\ph_{k',2i+1,0}^{t'}\mathcal{P}^{\PH\to \Phi}_{(k',i'),(2i+1,0)}, \quad t'\in \{\cos,\sin\},\;i' =0\ldots \frac{k'-1}{2} \\
\text{for } j' = 1\ldots &\frac{k'-1}{2}:\\
\phi^{t'}_{k',i',2j'} & = \sum\limits_{i=j'}^{\frac{k'-1}{2}}\ph_{k',2i,j'}^{t'}\mathcal{P}^{\PH\to \Phi}_{(k',i'),(2i,j')}, \quad t'\in\{\cos,\sin\},\;i' = j'\ldots \frac{k'-1}{2} \\
\phi^{t'}_{k',i',2j'+1} & = \sum\limits_{i=j'}^{\frac{k'-1}{2}}\ph_{k'2i+1,j'}^{t'}\mathcal{P}^{\PH\to \Phi}_{(k',i'),(2i-1,j'-1)}, \quad t'\in\{\cos,\sin\},\;i' = j'\ldots \frac{k'-1}{2}.
\end{align*}
In the beginning, there are 3 blocks of size $\tfrac{k'-1}{2}\times \tfrac{k'-1}{2}$. Then there are again 4 blocks of equal sizes $(\tfrac{k'-1}{2}-j+1)\times (\tfrac{k'-1}{2}-j+1),j=0\ldots \tfrac{k'-1}{2}$.
\par
For a fixed $k'$ the computational work is
$(\frac{k'}{2}+1)^2+4\sum_{j=1}^{\nicefrac{k'}{2}}(\frac{k'}{2}-j+1)^2 = \mathcal{O}(k'^3)$ for even $k'$ and for odd $k'$ this is $3(\frac{k'}{2}+1)^2+4\sum_{j=1}^{\lfloor\nicefrac{k'}{2}\rfloor}(\frac{k'-1}{2}-j+1)^2 = \mathcal{O}(k'^3)$. Summing over $k'=0\ldots N$ gives a transformation of complexity $\mathcal{O}(N^4)$.
\subsubsection{Multiplication with \texorpdfstring{$b_r$}{multwithbr}}
Now that we arrived at $f_2^{(\nicefrac{\vbar}{\sqrt{2}})}$ in the Spherical Laguerre basis, we consider the multiplication with $b_r(r)$. We define $f_{2,\beta}^{(\nicefrac{\vbar}{\sqrt{2}})}(\nicefrac{\vhat}{\sqrt{2}}):=b_r(\sqrt{2}r)f_2^{(\nicefrac{\vbar}{\sqrt{2}})}(\nicefrac{\vhat}{\sqrt{2}})$ and project onto $V_{N}$:
\begin{align}
\intrd e^{-r}(f_{2,r}^{(\frac{\vbar}{\sqrt{2}})})(\tfrac{\vhat}{\sqrt{2}})\Phi^{t'}_{k',i',j'}(\tfrac{\vhat}{\sqrt{2}})\,d\vhat = \intrd e^{-r}b_r(\sqrt{2}r)f_2^{(\frac{\vbar}{\sqrt{2}})}(\tfrac{\vhat}{\sqrt{2}})\Phi^{t'}_{k',i',j'}(\tfrac{\vhat}{\sqrt{2}})\,d\vhat.
\label{eq:multr}
\end{align}
Now, expand $f_2^{(\nicefrac{\vbar}{\sqrt{2}})}(\tfrac{\vhat}{\sqrt{2}}) = \sum \phi^t_{k,i,j}\Phi_{k,i,j}(\vhat)$ and $(f_{2,\beta}^{(\nicefrac{\vbar}{\sqrt{2}})})(\tfrac{\vhat}{\sqrt{2}}) = \sum {\phi^\beta}^t_{k,i,j}\Phi_{k,i,j}(\vhat)$, use $b_r(r) = r^{\beta}$ and plug everything into \refer{eq:multr} to arrive at:
\begin{equation*}
{\phi^\beta}^{t'}_{k',i',j'} = \sum\limits_{k} \phi^{t'}_{k,i',j'} \int\limits_{\R{+}} e^{-r}r^{2i'+0.5(1+\beta)}\mathcal{L}_{k-i'}^{2i+0.5}(r)\mathcal{L}_{k'-i'}^{2i'+0.5}(r)\,dr.
\end{equation*}
The integral on the right hand side is exactly evaluated by Gauss Laguerre quadratures with weight function $\omega(x) = e^{-r}r^{2i'+0.5(1+\beta)}$ \cite{Abramowitz_1964, shen2011spectral}.
\subsection{The transformation of the test functions}
In order to consider the transformation of the test functions, we denote by $\mathcal{P}^{N\to \Phi}:Q^N \to P^N$ the transformation from nodal to hierarchical polynomials. For $f \in P^N \subset Q^N$ there holds for each $\Phi \in P^N$
\begin{equation*}
\int e^{-|\vvec|^2}f(\vvec)\Phi(\vvec)\,d\vvec = \!\int e^{-|\vvec|^2}(\mathcal{P}^{N\to \Phi}f)(\vvec)\Phi(\vvec) \, d\vvec =\! \int e^{-|\vvec|^2}f(\vvec){\mathcal{P}^{{N\to \Phi}}}^*\Phi(\vvec) \, d\vvec. 
\end{equation*}
Thus, testing with Lagrange polynomials is obtained by applying the transposed matrix ${\mathcal{P}^{N\to \Phi}}^t$ to $Q^I(f^{\vbar}_2,\Phivec)$. Since $\mathcal{P}^{N \to \Phi} = \mathcal{P}^{\PH \to \Phi}\mathcal{P}^{H \to \PH}\mathcal{P}^{N \to H}$, the transposed matrix becomes  
\begin{equation*}
{\mathcal{P}^{N\to \Phi}}^t = {\mathcal{P}^{N\to H}}^t {\mathcal{P}^{H\to \PH}}^t{\mathcal{P}^{\PH\to \Phi}}^t.
\end{equation*}
As a consequence it has the same asymptotic computational costs as the forward transformation.
Finally, using the transformation of the test functions we obtain for the inner collision operator $Q^I$
\begin{equation*}
Q^I(f^{\nicefrac{\vbar}{\sqrt{2}}},\Lvec^{\nicefrac{\vbar}{\sqrt{2}}})(\nicefrac{\vbar}{\sqrt{2}}) = {\mathcal{P}^{N\to H}}^t {\mathcal{P}^{H\to \PH}}^t{\mathcal{P}^{\PH\to \Phi}}^t Q^I(f^{\nicefrac{\vbar}{\sqrt{2}}},\Phivec)(\nicefrac{\vbar}{\sqrt{2}}).
\end{equation*}
\subsection{The collision algorithm}
In this section we present the algorithm for the calculation of the collision operator.

\begin{algorithmic}
\STATE $N=\sqrt[3]{\ndofv}-1$
\STATE $(\vbar,\omega) = $ TensorGaussHermiteRule($N+1$)
\STATE $q = 0$
\FORALL{($\vbar$, $\omega$)}
	\STATE $c^{\vbar} = $ Shift$_{\vbar}$($c$)
	\FOR{$j = 0$ to $\ndofv^{(2)}-1$} \STATE $e_j = c^{\vbar}_{j}c^{\vbar}_{\ndofv^{(2)}-1-j}$\ENDFOR
	\STATE $h = $ Nodal2Hermite($c^{\vbar}$)
	\STATE $\ph = $ Hermite2CylinderHermite($h$)
	\STATE $\psi = $ CylinderHermite2SphericalLaguerre($\ph$)
	\STATE $\psi^{\beta} = $ Multr($\psi$)
	\STATE $p_\text{coll} = $ DiagCollision($\psi^{\beta}$)
	\STATE $\ph_\text{coll} = $ CylinderHermite2SphericalLaguerreT($\psi_{\text{coll}}$)
	\STATE $h_\text{coll} = $ Hermite2CylinderHermiteT($\ph_{\text{coll}}$)
	\STATE $n_\text{coll} = $ Nodal2HermiteT($h_{\text{coll}}$)\;
	\STATE $q += \omega\, $ShiftT$_{\vbar}$($n_\text{coll}$)\;
\ENDFOR
\end{algorithmic}
\begin{remark}
Note that testing in $V_N$ is obtained from $\intrd Q(f) \Lvec\,d\vvec$ via 
\begin{equation*}
\intrd Q(f) \Hvec\,d\vvec ={\mathcal{P}^{N\to H}}^{-t}\intrd Q(f) \Lvec\,d\vvec,
\end{equation*}
with $\Hvec$ denoting the vector of the Hermite polynomials.
\end{remark}
To summarize the computational complexity and memory requirements, we note that the complexity for each $\vbar$ is $\mathcal{O}(N^4)$. There are $N^3$ different $\vbar$ nodes such that the complexity is $\mathcal{O}(N^7)$ in total. A reduction of computational costs can be obtained by the use of low order integration rules w.r.t. $\vbar$. 
The complexity of our algorithm in terms of unknows per direction is higher than for Fourier-spectral methods \cite{doi:10.1137/16M1096001}, but the better approximation properties of the weighted Hermite polynomials outweigh the higher effort. This is documented in examples \ref{subsec:maxwell_moments} and  \ref{subsec:hs_moments}.
\par
The storage requirements are $\mathcal{O}(N^3)$ for the shift matrices and for the transformation from Hermite to Cylinder Hermite. $\mathcal{O}(N^2)$ is required for the transformation from Lagrange to Hermite polynomials and finally $\mathcal{O}(N^4)$ for the transformation from Cylinder Hermite to Spherical Laguerre. To take a concrete example, let $N=64$, then only 22 Megabytes of memory are required to store the transformations and the solution vector. This very low memory requirements are a direct consequence of reducing the transformations to several 1d transforms.
\par 
\section{Numerical results}
All our computations were done on Intel(R) Xeon(R) CPU E7-8867 v3 CPUs. The computation times for a single evaluation of the collision operator are documented in table \ref{tab:timings}. We note good scalability when the parallelization is done w.r.t. the $\vbar$ nodes.
\color{black}
\subsection{Maxwell molecules - BKW solution}
In a first example we consider a constant kernel $B(\vvec,\wvec,\evec') = \tfrac{1}{4\pi}$ and an initial condition $f_0(\vvec) = f_{\text{BKW}}(t_0,\vvec)$, where 
\begin{equation}
f_{\text{BKW}}(t,\vvec) = \frac{1}{2(2\pi K(t))^{\nicefrac{3}{2}}}\left(\frac{5K(t)-3}{K(t)} + \frac{1-K(t)}{K(t)^2}|\vvec|^2 \right) e^{-\frac{|\vvec|^2}{2K(t)}}.
\end{equation}
This is one of the few analytically available solutions of the non linear Boltzmann equation \cite{Krupp_1967,MR0398379,Ernst}.
$K(t)$ is given by $K(t) = 1-e^{-\nicefrac{t}{6}}$. To obtain a non negative initial condition, $t_0$ has to be greater than $6\ln(\nicefrac{5}{2})$, what is satisfied if $t_0 = 5.5$. For a given $t\geq t_0$, $f_{\text{BKW}}$ has a density $\rho(t)=1$, a mean velocity $V(t) = 0$ and a temperature $T=1$. Thus, to have the stationary solution exactly in our trial space we use $V_{2,(0,0,0)^t, N}$.
Figure \ref{fig:errors_bkw} shows the error w.r.t. time for different numbers of integration nodes w.r.t. $\vbar$. In sub figures \subref{fig:errors_bkwa} and \subref{fig:errors_bkwb} we use $\nip = N$ integration points for each direction to integrate w.r.t. $\vbar$. In sub figures \subref{fig:errors_bkwa075} and \subref{fig:errors_bkwb075} we chose less integration points w.r.t. the mean velocity, $\nip = 0.75 N$. This gives a similar result as for the exact integration w.r.t. $\vbar$. However, sub figures \subref{fig:errors_bkwa05} and \subref{fig:errors_bkwb05} show that a too low order quadrature for $\vbar$ yields worse convergence ($\nip = 0.5N$). In figure \ref{fig:errors_bkwc} we show the maximum of the errors from figure \ref{fig:errors_bkw} over time. This numerically confirms exponential convergence w.r.t. $N$.
We like to point out that in case of radially symmetric solutions, the fast algorithm proposed in \cite{doi:10.1137/16M1096001} is able to  exploit this symmetry and consequently becomes very efficient in that case. The errors at the same number of unknowns per direction are comparable.


\subsection{Maxwell molecules - moments} 
\label{subsec:maxwell_moments}Consider again the constant collision kernel $B = \tfrac{1}{4\pi}$. The initial condition is a sum of 2 Maxwellians:
\begin{equation}
f_0(\vvec) = \frac{\rho_1}{(2\pi)^{\nicefrac{3}{2}}} e^{\bigl|\frac{\vvec - V_1}{\sqrt{2T_1}}\bigr|^2}+\frac{\rho_2}{(2\pi)^{\nicefrac{3}{2}}} e^{\bigl|\frac{\vvec - V_2}{\sqrt{2T_2}}\bigr|^2}\
\label{eq:initial_2peaks}
\end{equation}
where $\rho_1 = \rho_2 = \tfrac{1}{2}, T_1 = T_2 = 1$ and $V_1 = (2,2,0)^t$ and $V_2 = (-2,0,0)^t$. This is an initial condition with density $\rho = 1,V=(0,1,0)^t$ and $T = \tfrac{8}{3}$. Consequently we use $V_{\tfrac{16}{3},(0,1,0)^t, N}$ for approximation. There is no exact solution known, but formulas for the momentum flow $P$ and the energy flux $q$ exist.
\begin{equation*}
P_{ij} = \intrd \vvec_i\vvec_jf(\vvec)\,d\vvec, \quad \quad q_i=\intrd \vvec_j |\vvec|^2f(\vvec)\,d\vvec.
\end{equation*}
For the initial condition \refer{eq:initial_2peaks} the non zero entries of these are given by
\begin{align*}
P_{11} &= \frac{7}{3}e^{-\nicefrac{t}{2}}+\frac{8}{3} &P_{22} & = -\frac{2}{3}e^{-\nicefrac{t}{2}}+\frac{11}{3} &\\
P_{33} &= -\frac{5}{3}e^{-\nicefrac{t}{2}}+\frac{8}{3} &P_{12} & = -2e^{-\nicefrac{t}{2}} &\\
q_1 &= -2e^{-\nicefrac{t}{2}} &q_{2} & = -\frac{2}{3}e^{-\nicefrac{t}{2}}+\frac{43}{6}. &\\
\end{align*}

We present the results in figures \ref{fig:Linfty_twopeak} and \ref{fig:moments_twopeaks_o8_10}. In figure \ref{fig:Linfty_twopeak}\subref{fig:linfty_maxwell} we depict the maximum error in the moments over time from which we deduce exponential convergence for the moments. Figure \ref{fig:moments_twopeaks_o8_10} shows the error in second and third order moments over time. Note that already very low expansion order $N=8$ gives a reasonable accuracy of $1e-4$ for the moments. 


\subsection{hard sphere molecules - moments}
\label{subsec:hs_moments} We consider again the initial condition \refer{eq:initial_2peaks}, but now for a different collision kernel 
\begin{equation*}
B(\vvec,\wvec,\evec') = \frac{1}{4\pi}|v-w|.
\end{equation*}
The moments and their errors are shown in figures \ref{fig:Linfty_twopeak} and \ref{fig:moments_twopeaks_hardspheres_o8}. The errors were calculated using a reference solution of order $N=22$ with a time step $dt=0.001$. The maximal errors over time are shown in figure \ref{fig:Linfty_twopeak}\subref{fig:linfty_hardspheres}, from which we again deduce exponential convergence of the method. A comparison with the results presented in \cite{doi:10.1137/16M1096001} shows also a good agreement.

\subsection{angular dependent collision kernel - moments} Now the collision kernel is given by
\begin{equation*}
B(\vvec,\wvec,\evec') = \frac{1}{4\pi}|v-w|^{0.38}(1+\cos(\theta))^{0.4}, \quad \text{with } \cos(\theta) = \frac{(\vvec-\wvec)\cdot \evec'}{|\vvec-\wvec|}.
\end{equation*}
We consider once more the initial condition \refer{eq:initial_2peaks}. The time evolution of the moments is depicted in figure \ref{fig:moments_twopeaks_o16_argon}. They are in good agreement with those published in \cite{doi:10.1137/16M1096001}.
We note that in the case of anisotropic solutions, the method proposed in \cite{doi:10.1137/16M1096001} requires a higher number of unknons per direction (33), to reach a similar error in the approximation of the moments than we do with order 9 polynomials. Without taking the constants into account, we require $9^7$ operations compared to $74\cdot 33^4$ operations.
\section{Conclusion}
In the present paper we developed an efficient algorithm for the Boltzmann collision operator in the unbounded velocity space. This enables us to include the collision invariants in the test space. Consequently, the conservation properties of the collision operator are naturally carried forward to the discrete level. For the evaluation of the collision operator we proposed an algorithm requiring $\mathcal{O}(N^7)$ floating operations and a storage of $\mathcal{O}(N^4)$. Almost all numerical tasks we need to perform are matrix- matrix multiplications allowing the usage of highly optimized Lapack routines \cite{laug}. Additionally, parallelization w.r.t. the $\vbar$ integration nodes is straight forward. Although there are faster methods available, the approximation with weighted Hermite polynomials shows high accuracy when evaluating moments of higher order, see sections \ref{subsec:maxwell_moments}, \ref{subsec:hs_moments}. 
\section{Acknowledgments}
Gerhard Kitzler is funded by the Austrian Science Fund (FWF) project F 65.
\tikzset{every picture/.style={scale=0.9}}%

\newpage

\begin{figure}[H]
	\centering
	\resizebox{0.4\linewidth}{!}{
\begin{tikzpicture}

\definecolor{color2}{rgb}{0.172549019607843,0.627450980392157,0.172549019607843}
\definecolor{color1}{rgb}{1,0.498039215686275,0.0549019607843137}
\definecolor{color0}{rgb}{0.12156862745098,0.466666666666667,0.705882352941177}
\definecolor{color3}{rgb}{0.83921568627451,0.152941176470588,0.156862745098039}

\begin{axis}[
xmin=6.8, xmax=33.2,
ymin=7.98909956392505e-08, ymax=0.0091686102622806,
ymode=log,
tick align=inside,
tick pos=left,
x grid style={white!69.019607843137251!black},
y grid style={white!69.019607843137251!black},
xlabel = {$N$},
ylabel = {$\| \,.\, \|_{L_\infty(0,T)}$},
legend entries={{$e_{L_2}(t),\, n_{\vbar} = N$},{$e_{L_\infty}(t),\,n_{\vbar} = N$},{$e_{L_2}(t),\, n_{\vbar} = 0.5N$ },{$e_{L_\infty}(t),\,n_{\vbar} = 0.5N$}},
legend style={at ={(0.05,-0.05)},anchor=south west, draw=black},
legend cell align={left}
]
\addlegendimage{no markers, color0}
\addlegendimage{no markers, color1}
\addlegendimage{no markers, color2}
\addlegendimage{no markers, color3}
\addplot [semithick, color0, forget plot]
table {%
8 0.00329446508918284
12 0.000712284957641958
16 0.000137644746771085
24 4.50090300708767e-06
32 1.35671616881966e-07
};
\addplot [semithick, color1, forget plot]
table {%
8 0.00539898778621552
12 0.00104651479926545
16 0.000193098523467006
24 6.0582925939215e-06
32 1.79150073925135e-07
};
\addplot [semithick, color2, forget plot]
table {%
8 0.0032944650891828086
12 0.001299062236122403
16 0.0006363929261167939
24 0.00011442378009659387
32 1.628349963795489e-05
};
\addplot [semithick, color3, forget plot]
table {%
8 0.00539898778621551
12 0.003070919044811642
16 0.0014191555838631283
24 0.00021852137924129614
32 2.748974343313502e-05
};

\end{axis}

\end{tikzpicture}}
	\caption{BKW solution. $\|e_{L_2}(t)\|_{L_\infty(5.5,8.5)}$, $\|e_{L_\infty}(t)\|_{L_\infty(5.5,8.5)}$ for different numbers $\nip$ of integration points for $\vbar$. Time stepping was done with Runge Kutta 4, $dt = 0.1$.}
	\label{fig:errors_bkwc}
\end{figure}
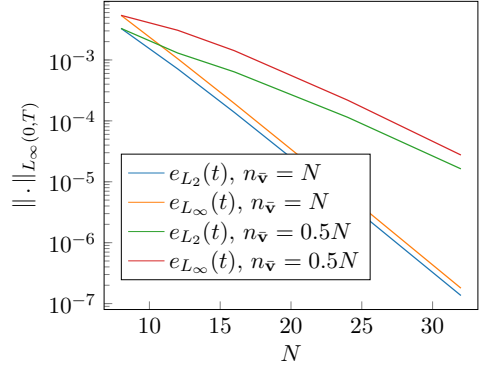

\begin{figure}[H]
	\centering
	\begin{subfigure}[b]{0.49\textwidth}
		\resizebox{0.86\linewidth}{!}{
\begin{tikzpicture}

\definecolor{color2}{rgb}{0.172549019607843,0.627450980392157,0.172549019607843}
\definecolor{color1}{rgb}{1,0.498039215686275,0.0549019607843137}
\definecolor{color4}{rgb}{0.580392156862745,0.403921568627451,0.741176470588235}
\definecolor{color5}{rgb}{0.549019607843137,0.337254901960784,0.294117647058824}
\definecolor{color0}{rgb}{0.12156862745098,0.466666666666667,0.705882352941177}
\definecolor{color3}{rgb}{0.83921568627451,0.152941176470588,0.156862745098039}

\begin{axis}[
xlabel={N},
ylabel={$\|\,.\,\|_{L_\infty(0,6)}$},
xmin=8, xmax=24,
ymin=1e-14, ymax=0.0005,
xmode=log,
ymode=log,
tick align=inside,
tick pos=left,
x grid style={white!69.019607843137251!black},
y grid style={white!69.019607843137251!black},
xtick = {8,12,16,24},
xticklabels = {8,12,16,24},
legend entries = {{$P_{11}$},{$P_{22}$},{$P_{33}$},{$P_{12}$},{$q_{1}$},{$q_{2}$}},
legend style={at={(1,0.9)},anchor=north west,nodes=right,draw = black}
]

\addlegendimage{no markers, color0}
\addlegendimage{no markers, color1}
\addlegendimage{no markers, color2}
\addlegendimage{no markers, color3}
\addlegendimage{no markers, color4}
\addlegendimage{no markers, color5}

\addplot [semithick, color0, forget plot]
table {%
8 9.72031810038132e-05
12 2.25868948433572e-07
16 2.91635338101059e-09
24 4.5772274859246e-12
};
\addplot [semithick, color1, forget plot]
table {%
8 1.76972796102071e-05
12 2.3155115602691e-07
16 6.71748878744438e-10
24 1.31494815036604e-12
};
\addplot [semithick, color2, forget plot]
table {%
8 6.22239432982452e-05
12 3.48243861081343e-07
16 2.56109489171763e-09
24 3.25517390820096e-12
};
\addplot [semithick, color3, forget plot]
table {%
8 4.29643366905985e-05
12 7.45708246263632e-08
16 7.08200165178141e-10
24 3.89699383873676e-12
};
\addplot [semithick, color4, forget plot]
table {%
8 4.29643366826049e-05
12 7.4570825292497e-08
16 7.08196168375252e-10
24 3.90165677544019e-12
};
\addplot [semithick, color5, forget plot]
table {%
8 1.28046594616293e-05
12 5.20534577219678e-07
16 8.29902369048341e-10
24 1.31628041799559e-12
};
\end{axis}

\end{tikzpicture}}
		\subcaption{}
		\label{fig:linfty_maxwell}
	\end{subfigure}
	\hfill
	\begin{subfigure}[b]{0.49\textwidth}
		\resizebox{0.86\linewidth}{!}{
\begin{tikzpicture}

\definecolor{color2}{rgb}{0.172549019607843,0.627450980392157,0.172549019607843}
\definecolor{color1}{rgb}{1,0.498039215686275,0.0549019607843137}
\definecolor{color4}{rgb}{0.580392156862745,0.403921568627451,0.741176470588235}
\definecolor{color5}{rgb}{0.549019607843137,0.337254901960784,0.294117647058824}
\definecolor{color0}{rgb}{0.12156862745098,0.466666666666667,0.705882352941177}
\definecolor{color3}{rgb}{0.83921568627451,0.152941176470588,0.156862745098039}

\begin{axis}[
xlabel={N},
ylabel={$\|\,.\,\|_{L_\infty(0,6)}$},
xmin=8, xmax=24,
ymin=1e-08, ymax=0.0005,
xmode=log,
ymode=log,
tick align=inside,
tick pos=left,
x grid style={white!69.019607843137251!black},
y grid style={white!69.019607843137251!black},
xtick = {8,12,16,24},
xticklabels = {8,12,16,24},
legend entries = {{$P_{11}$},{$P_{22}$},{$P_{33}$},{$P_{12}$},{$q_{1}$},{$q_{2}$}},
legend style={at={(1,0.9)},anchor=north west,nodes=right,draw = black}
]

\addlegendimage{no markers, color0}
\addlegendimage{no markers, color1}
\addlegendimage{no markers, color2}
\addlegendimage{no markers, color3}
\addlegendimage{no markers, color4}
\addlegendimage{no markers, color5}

\addplot [semithick, color0, forget plot]
table {%
8 0.000140394654084108
12 1.25808430420093e-05
16 1.4610405236759e-06
20 1.2525400716612e-07
};
\addplot [semithick, color1, forget plot]
table {%
8 0.000289143856125218
12 2.44753798335218e-05
16 2.75217326572985e-06
20 2.3020777151217e-07
};
\addplot [semithick, color2, forget plot]
table {%
8 0.000167313537003722
12 1.87069830519526e-05
16 1.92820766020674e-06
20 1.67477695711682e-07
};
\addplot [semithick, color3, forget plot]
table {%
8 4.29643366923749e-05
12 2.13359549894321e-06
16 1.86371160296184e-07
20 2.9893331476849e-08
};
\addplot [semithick, color4, forget plot]
table {%
8 4.29643366834931e-05
12 2.13359549405823e-06
16 1.86371162336219e-07
20 2.98933408027224e-08
};
\addplot [semithick, color5, forget plot]
table {%
8 0.000284251235956212
12 2.47643632302896e-05
16 2.75233139390707e-06
20 2.30207477969202e-07
};
\end{axis}

\end{tikzpicture}}
		\subcaption{}
		\label{fig:linfty_hardspheres}
	\end{subfigure}
\caption{Maxwellian sum. The maximum of the error in the moments over time for Maxwell molecules \subref{fig:linfty_maxwell}, and hard spheres \subref{fig:linfty_hardspheres}.}
\label{fig:Linfty_twopeak}
\end{figure}
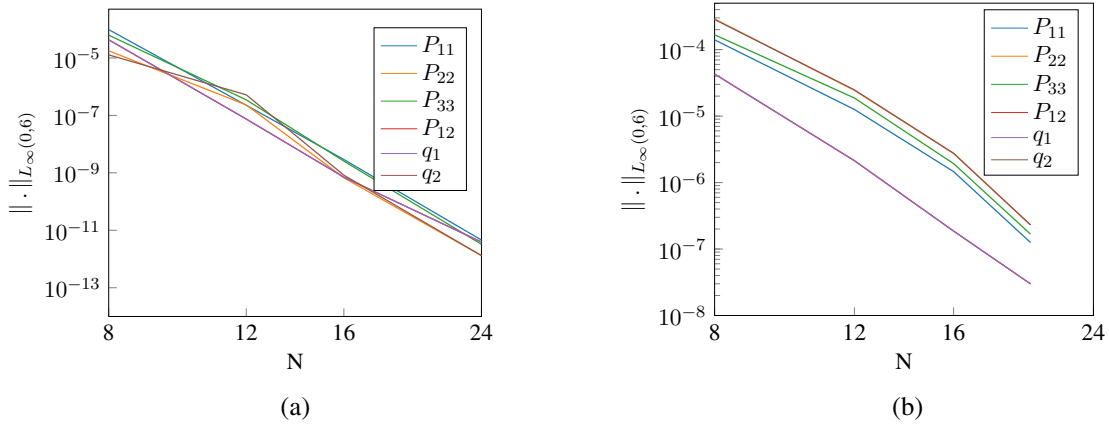

\begin{table}
\centering
  \begin{tabular}{cccc}
   $N$ & 1 thread & 8 threads & 16 threads \\ \hline
   4 &   0.0013 &  0.0009 (1.44) & 0.0008 (1.63) \\
   8 &   0.1129 &  0.0189 (5.97) & 0.0092 (12.27)\\
  16 &  10.3471 &  1.3001 (7.96) & 0.7255 (14.26)\\
  24 & 153.2798 & 19.2243 (7.97) &10.3800 (14.77)\\
  32 &1196.0663 &158.7430 (7.53) &99.0189 (12.10)\\
  \end{tabular}
\caption{Computation times [s] for a single application of the collision operator. The values in brackets show the speedup w.r.t. 1 thread.}
\label{tab:timings}
\end{table}

\begin{figure}[H]
    \centering
    \begin{subfigure}[b]{0.49\textwidth}
		\resizebox{0.925\linewidth}{!}{
\begin{tikzpicture}

\definecolor{color2}{rgb}{0.172549019607843,0.627450980392157,0.172549019607843}
\definecolor{color1}{rgb}{1,0.498039215686275,0.0549019607843137}
\definecolor{color4}{rgb}{0.580392156862745,0.403921568627451,0.741176470588235}
\definecolor{color0}{rgb}{0.12156862745098,0.466666666666667,0.705882352941177}
\definecolor{color3}{rgb}{0.83921568627451,0.152941176470588,0.156862745098039}

\begin{axis}[
xmin=5.355, xmax=8.54499999999999,
ymin=1e-9, ymax=0.01,
ymode=log,
tick align=inside,
tick pos=left,
x grid style={white!69.019607843137251!black},
y grid style={white!69.019607843137251!black},
ylabel = {$e_{L_\infty}(t)$},
xlabel = {time},
]

\addplot [semithick, color0, forget plot]
table {%
5.5 0.00329446508918284
5.6 0.00296927807129354
5.7 0.00267830925051757
5.8 0.00241772793568754
5.9 0.00218415802013787
6 0.00197462098949998
6.1 0.0017864864612648
6.2 0.00161742920820884
6.3 0.00146539177116272
6.4 0.00132855189638478
6.5 0.0012052941427964
6.6 0.00109418509770801
6.7 0.000993951719053615
6.8 0.000903462389750607
6.9 0.000821710327452749
6.99999999999999 0.000747799042197461
7.09999999999999 0.000680929576558451
7.19999999999999 0.000620389298975829
7.29999999999999 0.000565542051864251
7.39999999999999 0.000515819482656714
7.49999999999999 0.000470713408775962
7.59999999999999 0.000429769087183909
7.69999999999999 0.000392579276102679
7.79999999999999 0.000358778991123826
7.89999999999999 0.000328040870553235
7.99999999999999 0.000300071075767364
8.09999999999999 0.000274605661815544
8.19999999999999 0.000251407361705307
8.29999999999999 0.00023026273492293
8.39999999999999 0.000210979636922272
};
\addplot [semithick, color1, forget plot]
table {%
5.5 0.000712284957641958
5.6 0.000619769597909397
5.7 0.000539736910481684
5.8 0.000470439514137413
5.9 0.000410383694193672
6 0.000358291203768648
6.1 0.000313067036029031
6.2 0.000273772197747837
6.3 0.00023960067823291
6.4 0.000209859942658156
6.5 0.000183954390293312
6.6 0.000161371310345264
6.7 0.000141668944550345
6.8 0.000124466329103082
6.9 0.000109434641262812
6.99999999999999 9.62898199157874e-05
7.09999999999999 8.47862660197651e-05
7.19999999999999 7.47114594695092e-05
7.29999999999999 6.58813545332607e-05
7.39999999999999 5.81364374637044e-05
7.49999999999999 5.13383478841084e-05
7.59999999999999 4.53669806673718e-05
7.69999999999999 4.01179977412638e-05
7.79999999999999 3.55006899594029e-05
7.89999999999999 3.14361382056094e-05
7.99999999999999 2.7855630520518e-05
8.09999999999999 2.46992984792582e-05
8.19999999999999 2.19149414987449e-05
8.29999999999999 1.94570123683568e-05
8.39999999999999 1.7285741211695e-05
};
\addplot [semithick, color2, forget plot]
table {%
5.5 0.000137644746771085
5.6 0.000115863755218862
5.7 9.76198340431889e-05
5.8 8.23244586550566e-05
5.9 6.94894346631135e-05
6 5.87093582427971e-05
6.1 4.96472248805516e-05
6.2 4.20226054162719e-05
6.3 3.56019187216713e-05
6.4 3.01904191450583e-05
6.5 2.56255883856506e-05
6.6 2.17716791979586e-05
6.7 1.85152050068763e-05
6.8 1.57612073211633e-05
6.9 1.34301635026923e-05
6.99999999999999 1.14554223665728e-05
7.09999999999999 9.78107536739325e-06
7.19999999999999 8.36018765038721e-06
7.29999999999999 7.15332673467683e-06
7.39999999999999 6.12733761826828e-06
7.49999999999999 5.2543221150253e-06
7.59999999999999 4.51078762570001e-06
7.69999999999999 3.87693661031696e-06
7.79999999999999 3.33607301167965e-06
7.89999999999999 2.87410597757387e-06
7.99999999999999 2.47913460328653e-06
8.09999999999999 2.14110019857577e-06
8.19999999999999 1.85149487814951e-06
8.29999999999999 1.60311717131018e-06
8.39999999999999 1.38986691538417e-06
};
\addplot [semithick, color3, forget plot]
table {%
5.5 4.50090300708767e-06
5.6 3.54820485692041e-06
5.7 2.80006925561585e-06
5.8 2.21202927482532e-06
5.9 1.74939952510972e-06
6 1.3850978392918e-06
6.1 1.09795847650848e-06
6.2 8.71424356330835e-07
6.3 6.92531972559807e-07
6.4 5.51122604355164e-07
6.5 4.39228716345857e-07
6.6 3.5059614491087e-07
6.7 2.80311651688903e-07
6.8 2.24512329977111e-07
6.9 1.80158665680219e-07
6.99999999999999 1.44857150323474e-07
7.09999999999999 1.167215073275e-07
7.19999999999999 9.42640348070012e-08
7.29999999999999 7.63104605200349e-08
7.39999999999999 6.1933168434003e-08
7.49999999999999 5.0398793104145e-08
7.59999999999999 4.11270591957336e-08
7.69999999999999 3.36584300431042e-08
7.79999999999999 2.76286615877308e-08
7.89999999999999 2.27487739409837e-08
7.99999999999999 1.87892761627229e-08
8.09999999999999 1.556773246425e-08
8.19999999999999 1.29389558720105e-08
8.29999999999999 1.07872685922716e-08
8.39999999999999 9.02038935907123e-09
};
\addplot [semithick, color4, forget plot]
table {%
5.5 1.35671616881966e-07
5.6 9.93321501857602e-08
5.7 7.25447627959036e-08
5.8 5.27806203153014e-08
5.9 3.81864701387569e-08
6 2.74026516443315e-08
6.1 1.99778066936607e-08
6.2 1.50689049155384e-08
6.3 1.14356854680198e-08
6.4 8.74341037071513e-09
6.5 6.74551790813549e-09
6.6 5.26030392161703e-09
6.7 4.15379761142809e-09
6.8 3.32717833839702e-09
6.9 2.70752307379718e-09
6.99999999999999 2.24099917095488e-09
7.09999999999999 1.88785281851622e-09
7.19999999999999 2.1054363322226e-09
7.29999999999999 2.25434603623453e-09
7.39999999999999 2.34294046663752e-09
7.49999999999999 2.38652560213826e-09
7.59999999999999 2.39661141626657e-09
7.69999999999999 2.38187333051476e-09
7.79999999999999 2.34886545630308e-09
7.89999999999999 2.30255038646821e-09
7.99999999999999 2.24669430975988e-09
8.09999999999999 2.18416087299778e-09
8.19999999999999 2.11713145625847e-09
8.29999999999999 2.04726959690538e-09
8.39999999999999 1.97584459904476e-09
};
\end{axis}

\end{tikzpicture}}
		\subcaption{}
		\label{fig:errors_bkwa}
    \end{subfigure}
    \hfill
    \begin{subfigure}[b]{0.49\textwidth}
		\resizebox{0.925\linewidth}{!}{
\begin{tikzpicture}

\definecolor{color2}{rgb}{0.172549019607843,0.627450980392157,0.172549019607843}
\definecolor{color1}{rgb}{1,0.498039215686275,0.0549019607843137}
\definecolor{color4}{rgb}{0.580392156862745,0.403921568627451,0.741176470588235}
\definecolor{color0}{rgb}{0.12156862745098,0.466666666666667,0.705882352941177}
\definecolor{color3}{rgb}{0.83921568627451,0.152941176470588,0.156862745098039}

\begin{axis}[
xmin=5.355, xmax=8.54499999999999,
ymin=1e-9, ymax=0.01,
ymode=log,
tick align=inside,
tick pos=left,
x grid style={white!69.019607843137251!black},
y grid style={white!69.019607843137251!black},
ylabel = {$e_{L_2}(t)$},
xlabel = {time},
]

\addplot [semithick, color0, forget plot]
table {%
5.5 0.00539898778621552
5.6 0.00488192453777078
5.7 0.00441744386816675
5.8 0.00399984162104235
5.9 0.00362407553609227
6 0.00328568308260357
6.1 0.00298071017920136
6.2 0.00270564926640188
6.3 0.00245738542726599
6.4 0.00223314944415322
6.5 0.00203047684225002
6.6 0.00184717210811729
6.7 0.00168127738804175
6.8 0.00153104506986774
6.9 0.00139491373604197
6.99999999999999 0.00127148704716184
7.09999999999999 0.00115951517633439
7.19999999999999 0.00105787846675864
7.29999999999999 0.000965573029507998
7.39999999999999 0.000881698036658441
7.49999999999999 0.000805444497646694
7.59999999999999 0.000736085334862783
7.69999999999999 0.000672966598668494
7.79999999999999 0.000615499682864243
7.89999999999999 0.000563154419591841
7.99999999999999 0.000515452948175248
8.09999999999999 0.000471964265813684
8.19999999999999 0.000432299379655828
8.29999999999999 0.000396106989848364
8.39999999999999 0.00036306964188868
};
\addplot [semithick, color1, forget plot]
table {%
5.5 0.00104651479926545
5.6 0.000916162871070042
5.7 0.000802597208116837
5.8 0.000703579386036592
5.9 0.000617180921049217
6 0.000541738571136129
6.1 0.000475816370427914
6.2 0.000418173335091584
6.3 0.000367735954481445
6.4 0.000323574726385248
6.5 0.000284884115376938
6.6 0.000250965413052344
6.7 0.000221212061907298
6.8 0.000195097073774198
6.9 0.000172162231474614
6.99999999999999 0.000152008810642079
7.09999999999999 0.000134289599138894
7.19999999999999 0.00011870202545601
7.29999999999999 0.000104982236042022
7.39999999999999 9.28999855526942e-05
7.49999999999999 8.22542242912617e-05
7.59999999999999 7.28692842384954e-05
7.69999999999999 6.45915795591446e-05
7.79999999999999 5.72867497442164e-05
7.89999999999999 5.08371839570894e-05
7.99999999999999 4.5139873993794e-05
8.09999999999999 4.01045507862758e-05
8.19999999999999 3.5652065780203e-05
8.29999999999999 3.1712983977381e-05
8.39999999999999 2.82263600925725e-05
};
\addplot [semithick, color2, forget plot]
table {%
5.5 0.000193098523467006
5.6 0.00016360965425303
5.7 0.0001387209821771
5.8 0.000117699853107343
5.9 9.99328773305469e-05
6 8.49060476121929e-05
6.1 7.21882597306201e-05
6.2 6.14176361400562e-05
6.3 5.22901623772667e-05
6.4 4.45502342930478e-05
6.5 3.79827861571406e-05
6.6 3.24067283421878e-05
6.7 2.76694711873081e-05
6.8 2.36423508158366e-05
6.9 2.0216804778359e-05
6.99999999999999 1.73011717293627e-05
7.09999999999999 1.48180109926016e-05
7.19999999999999 1.27018556862866e-05
7.29999999999999 1.0897327763924e-05
7.39999999999999 9.3575554470501e-06
7.49999999999999 8.04284354217242e-06
7.59999999999999 6.91955542203194e-06
7.69999999999999 5.95917231831772e-06
7.79999999999999 5.13750126691646e-06
7.89999999999999 4.43400777631303e-06
7.99999999999999 3.83125321970555e-06
8.09999999999999 3.31442022218322e-06
8.19999999999999 2.87091203830434e-06
8.29999999999999 2.49001418608915e-06
8.39999999999999 2.16260849893027e-06
};
\addplot [semithick, color3, forget plot]
table {%
5.5 6.0582925939215e-06
5.6 4.80358972375019e-06
5.7 3.81162667898571e-06
5.8 3.02686630027076e-06
5.9 2.40562026039914e-06
6 1.91349621462482e-06
6.1 1.52340086905755e-06
6.2 1.21397647391173e-06
6.3 9.68375593032112e-07
6.4 7.73300152979455e-07
6.5 6.18247159257492e-07
6.6 4.94916173808351e-07
6.7 3.96743516945138e-07
6.8 3.18535828394397e-07
6.9 2.56181594698465e-07
6.99999999999999 2.0642390352095e-07
7.09999999999999 1.66681315106372e-07
7.19999999999999 1.34906575780476e-07
7.29999999999999 1.09475113170798e-07
7.39999999999999 8.90969858550227e-08
7.49999999999999 7.27473169213466e-08
7.59999999999999 5.96113037412799e-08
7.69999999999999 4.90407301317748e-08
7.79999999999999 4.05195607632118e-08
7.89999999999999 3.36367110694254e-08
7.99999999999999 2.80644888920739e-08
8.09999999999999 2.35415202460557e-08
8.19999999999999 1.9859220577961e-08
8.29999999999999 1.68510681864493e-08
8.39999999999999 1.43840912778483e-08
};
\addplot [semithick, color4, forget plot]
table {%
5.5 1.79150073925135e-07
5.6 1.32841632898751e-07
5.7 9.8577017539245e-08
5.8 7.32110853523535e-08
5.9 5.44252317084227e-08
6 4.05093784914784e-08
6.1 3.02018366111569e-08
6.2 2.25715787031494e-08
6.3 1.69315894882222e-08
6.4 1.27749518759396e-08
6.5 9.72747888544614e-09
6.6 7.51225602587374e-09
6.7 5.92267423192289e-09
6.8 4.80173534392327e-09
6.9 4.02679714385125e-09
6.99999999999999 3.49994104689858e-09
7.09999999999999 3.14355120135503e-09
7.19999999999999 2.89875646015412e-09
7.29999999999999 2.72376114826021e-09
7.39999999999999 2.59083365898292e-09
7.49999999999999 2.48259265973038e-09
7.59999999999999 2.38861348167735e-09
7.69999999999999 2.30285789953726e-09
7.79999999999999 2.22193823199718e-09
7.89999999999999 2.14402288881896e-09
7.99999999999999 2.06817386309498e-09
8.09999999999999 1.9939552462098e-09
8.19999999999999 1.92120564186376e-09
8.29999999999999 1.84990712235137e-09
8.39999999999999 1.78011099556506e-09
};
\end{axis}

\end{tikzpicture}}
		\subcaption{}
		\label{fig:errors_bkwb}		
    \end{subfigure}

    \begin{subfigure}[b]{0.49\textwidth}
		\resizebox{0.925\linewidth}{!}{
\begin{tikzpicture}

\definecolor{color2}{rgb}{0.172549019607843,0.627450980392157,0.172549019607843}
\definecolor{color1}{rgb}{1,0.498039215686275,0.0549019607843137}
\definecolor{color4}{rgb}{0.580392156862745,0.403921568627451,0.741176470588235}
\definecolor{color0}{rgb}{0.12156862745098,0.466666666666667,0.705882352941177}
\definecolor{color3}{rgb}{0.83921568627451,0.152941176470588,0.156862745098039}

\begin{axis}[
xmin=5.355, xmax=8.54499999999999,
ymin=1e-9, ymax=0.01,
ymode=log,
tick align=inside,
tick pos=left,
x grid style={white!69.019607843137251!black},
y grid style={white!69.019607843137251!black},
ylabel = {$e_{L_\infty}(t)$},
xlabel = {time},
]

\addplot [semithick, color0, forget plot]
table {%
5.5 0.00329447
5.6 0.00295272
5.7 0.00264719
5.8 0.00237374
5.9 0.00212877
6 0.00190907
6.1 0.00171187
6.2 0.00153471
6.3 0.00137541
6.4 0.00123206
6.5 0.00110299
6.6 0.000986684
6.7 0.000881828
6.8 0.000787244
6.9 0.000701889
7 0.000624833
7.1 0.000555249
7.2 0.000492397
7.3 0.000435619
7.4 0.000384325
7.5 0.000337985
7.6 0.000296125
7.7 0.000258321
7.8 0.000224189
7.9 0.000193385
8 0.000165598
8.1 0.000140549
8.2 0.000117985
8.3 9.76763e-05
8.4 7.9417e-05
};
\addplot [semithick, color1, forget plot]
table {%
5.5 0.000712285
5.6 0.000637623
5.7 0.0005722
5.8 0.000514748
5.9 0.000464183
6 0.000419578
6.1 0.00038014
6.2 0.000345188
6.3 0.000314138
6.4 0.000286489
6.5 0.000261809
6.6 0.000239727
6.7 0.000219921
6.8 0.000202116
6.9 0.000186071
7 0.000171578
7.1 0.000158457
7.2 0.000146553
7.3 0.000135729
7.4 0.000125865
7.5 0.000116859
7.6 0.000108619
7.7 0.000101066
7.8 9.41303e-05
7.9 8.77501e-05
8 8.18714e-05
8.1 7.64462e-05
8.2 7.14322e-05
8.3 6.67915e-05
8.4 6.24909e-05
};
\addplot [semithick, color2, forget plot]
table {%
5.5 0.000137645
5.6 0.000109121
5.7 8.56638e-05
5.8 6.64031e-05
5.9 5.06176e-05
  6 3.77108e-05
6.1 2.71883e-05
6.2 1.86402e-05
6.3 1.19138e-05
6.4 8.87589e-06
6.5 8.05957e-06
6.6 8.27033e-06
6.7 8.3513e-06
6.8 8.32793e-06
6.9 8.31079e-06
  7 9.50756e-06
7.1 1.03599e-05
7.2 1.0932e-05
7.3 1.12773e-05
7.4 1.14398e-05
7.5 1.14563e-05
7.6 1.13572e-05
7.7 1.11675e-05
7.8 1.09079e-05
7.9 1.05956e-05
  8 1.02444e-05
8.1 9.86586e-06
8.2 9.46928e-06
8.3 9.06223e-06
8.4 8.6508e-06
};
\addplot [semithick, color3, forget plot]
table {%
5.5 4.5009e-06
5.6 3.23498e-06
5.7 2.26131e-06
5.8 1.51549e-06
5.9 9.47221e-07
  6 5.94802e-07
6.1 3.36291e-07
6.2 2.69133e-07
6.3 3.07121e-07
6.4 3.42437e-07
6.5 4.28013e-07
6.6 4.84043e-07
6.7 5.17558e-07
6.8 5.34051e-07
6.9 5.37803e-07
  7 5.32148e-07
7.1 5.19676e-07
7.2 5.02389e-07
7.3 4.81832e-07
7.4 4.5919e-07
7.5 4.35366e-07
7.6 4.11041e-07
7.7 3.86725e-07
7.8 3.62793e-07
7.9 3.39513e-07
  8 3.17076e-07
8.1 2.95606e-07
8.2 2.75183e-07
8.3 2.55847e-07
8.4 2.37612e-07

};
\addplot [semithick, color4, forget plot]
table {%
5.5 1.35672e-07
5.6 8.9502e-08
5.7 5.60799e-08
5.8 3.37624e-08
5.9 1.97993e-08
  6 9.86299e-09
6.1 7.36812e-09
6.2 1.14274e-08
6.3 1.51768e-08
6.4 1.74954e-08
6.5 1.87824e-08
6.6 1.93321e-08
6.7 1.93614e-08
6.8 1.90295e-08
6.9 1.84534e-08
  7 1.77183e-08
7.1 1.68866e-08
7.2 1.60028e-08
7.3 1.50992e-08
7.4 1.41985e-08
7.5 1.33161e-08
7.6 1.24629e-08
7.7 1.16454e-08
7.8 1.0868e-08
7.9 1.01326e-08
  8 9.43995e-09
8.1 8.78978e-09
8.2 8.1811e-09
8.3 7.61243e-09
8.4 7.08203e-09
};
\end{axis}

\end{tikzpicture}}
		\subcaption{}
		\label{fig:errors_bkwa075}
    \end{subfigure}
    \hfill
    \begin{subfigure}[b]{0.49\textwidth}
		\resizebox{0.925\linewidth}{!}{
\begin{tikzpicture}

\definecolor{color2}{rgb}{0.172549019607843,0.627450980392157,0.172549019607843}
\definecolor{color1}{rgb}{1,0.498039215686275,0.0549019607843137}
\definecolor{color4}{rgb}{0.580392156862745,0.403921568627451,0.741176470588235}
\definecolor{color0}{rgb}{0.12156862745098,0.466666666666667,0.705882352941177}
\definecolor{color3}{rgb}{0.83921568627451,0.152941176470588,0.156862745098039}

\begin{axis}[
xmin=5.355, xmax=8.54499999999999,
ymin=1e-9, ymax=0.01,
ymode=log,
tick align=inside,
tick pos=left,
x grid style={white!69.019607843137251!black},
y grid style={white!69.019607843137251!black},
ylabel = {$e_{L_2}(t)$},
xlabel = {time},
]

\addplot [semithick, color0, forget plot]
table {%
5.5 0.00539899
5.6 0.00482908
5.7 0.00431993
5.8 0.00386504
5.9 0.00345864
  6 0.0030956
6.1 0.00277135
6.2 0.00248184
6.3 0.00222343
6.4 0.00199291
6.5 0.00178739
6.6 0.00160428
6.7 0.0014413
6.8 0.00129636
6.9 0.00116762
  7 0.00105341
7.1 0.000952245
7.2 0.000862761
7.3 0.00078374
7.4 0.000714075
7.5 0.000652756
7.6 0.000598866
7.7 0.000551563
7.8 0.000510079
7.9 0.000473712
  8 0.000441821
8.1 0.000413824
8.2 0.000389194
8.3 0.000367459
8.4 0.000348199

};
\addplot [semithick, color1, forget plot]
table {%
5.5 0.00104651
5.6 0.000944528
5.7 0.000855184
5.8 0.000776635
5.9 0.000707319
  6 0.000645914
6.1 0.000591306
6.2 0.000542551
6.3 0.000498855
6.4 0.000459545
6.5 0.000424054
6.6 0.000391901
6.7 0.000362678
6.8 0.00033604
6.9 0.00031169
  7 0.000289377
7.1 0.000268883
7.2 0.000250022
7.3 0.000232632
7.4 0.000216573
7.5 0.00020172
7.6 0.000187966
7.7 0.000175215
7.8 0.000163383
7.9 0.000152394
  8 0.00014218
8.1 0.00013268
8.2 0.00012384
8.3 0.00011561
8.4 0.000107944

};
\addplot [semithick, color2, forget plot]
table {%
5.5 0.000193099
5.6 0.000154853
5.7 0.000123417
5.8 9.77494e-05
5.9 7.70016e-05
  6 6.04949e-05
6.1 4.76984e-05
6.2 3.81969e-05
6.3 3.16246e-05
6.4 2.75572e-05
6.5 2.54277e-05
6.6 2.45789e-05
6.7 2.44214e-05
6.8 2.45375e-05
6.9 2.46809e-05
  7 2.47265e-05
7.1 2.46229e-05
7.2 2.43592e-05
7.3 2.39449e-05
7.4 2.3399e-05
7.5 2.27441e-05
7.6 2.20027e-05
7.7 2.11958e-05
7.8 2.03423e-05
7.9 1.94586e-05
  8 1.85587e-05
8.1 1.76542e-05
8.2 1.67547e-05
8.3 1.5868e-05
8.4 1.50002e-05

};
\addplot [semithick, color3, forget plot]
table {%
5.5 6.05829e-06
5.6 4.43214e-06
5.7 3.1858e-06
5.8 2.24587e-06
5.9 1.56002e-06
  6 1.09621e-06
6.1 8.36535e-07
6.2 7.48149e-07
6.3 7.60657e-07
6.4 8.06021e-07
6.5 8.49062e-07
6.6 8.77752e-07
6.7 8.90113e-07
6.8 8.87788e-07
6.9 8.73462e-07
  7 8.49878e-07
7.1 8.19488e-07
7.2 7.84353e-07
7.3 7.46149e-07
7.4 7.06206e-07
7.5 6.65562e-07
7.6 6.25009e-07
7.7 5.85144e-07
7.8 5.46402e-07
7.9 5.09093e-07
  8 4.73425e-07
8.1 4.3953e-07
8.2 4.07478e-07
8.3 3.77294e-07
8.4 3.48968e-07
};
\addplot [semithick, color4, forget plot]
table {%
5.5 1.7915e-07
5.6 1.21421e-07
5.7 7.98794e-08
5.8 5.06761e-08
5.9 3.13867e-08
  6 2.10058e-08
6.1 1.84798e-08
6.2 2.0015e-08
6.3 2.21525e-08
6.4 2.36996e-08
6.5 2.44845e-08
6.6 2.4612e-08
6.7 2.42373e-08
6.8 2.35032e-08
6.9 2.2526e-08
  7 2.13961e-08
7.1 2.01809e-08
7.2 1.89303e-08
7.3 1.768e-08
7.4 1.64549e-08
7.5 1.52722e-08
7.6 1.41428e-08
7.7 1.30736e-08
7.8 1.20678e-08
7.9 1.11266e-08
  8 1.02496e-08
8.1 9.435e-09
8.2 8.68045e-09
8.3 7.98305e-09
8.4 7.33962e-09
};
\end{axis}

\end{tikzpicture}}
		\subcaption{}
		\label{fig:errors_bkwb075}		
    \end{subfigure}
    
    \begin{subfigure}[b]{0.49\textwidth}
		\resizebox{0.925\linewidth}{!}{
\begin{tikzpicture}

\definecolor{color2}{rgb}{0.172549019607843,0.627450980392157,0.172549019607843}
\definecolor{color1}{rgb}{1,0.498039215686275,0.0549019607843137}
\definecolor{color4}{rgb}{0.580392156862745,0.403921568627451,0.741176470588235}
\definecolor{color0}{rgb}{0.12156862745098,0.466666666666667,0.705882352941177}
\definecolor{color3}{rgb}{0.83921568627451,0.152941176470588,0.156862745098039}

\begin{axis}[
xmin=5.355, xmax=8.54499999999999,
ymin=1e-9, ymax=0.01,
ymode=log,
tick align=inside,
tick pos=left,
x grid style={white!69.019607843137251!black},
y grid style={white!69.019607843137251!black},
ylabel = {$e_{L_\infty}(t)$},
xlabel = {time},
]

\addplot [semithick, color0, forget plot]
table {%
5.5 0.00329447
5.6 0.0031895
5.7 0.00306258
5.8 0.00291664
5.9 0.00275452
  6 0.00257899
6.1 0.00239262
6.2 0.00219784
6.3 0.00199685
6.4 0.00179167
6.5 0.0015841
6.6 0.00137575
6.7 0.00116802
6.8 0.000964013
6.9 0.00097358
  7 0.000976882
7.1 0.000994632
7.2 0.00101791
7.3 0.00105201
7.4 0.00109484
7.5 0.00113253
7.6 0.00116552
7.7 0.00119421
7.8 0.00121896
7.9 0.00124011
  8 0.00125797
8.1 0.00127281
8.2 0.00135426
8.3 0.00146902
8.4 0.00157703

};
\addplot [semithick, color1, forget plot]
table {%
5.5 0.000712285
5.6 0.000586579
5.7 0.000467012
5.8 0.000352686
5.9 0.000367202
  6 0.000443389
6.1 0.000509635
6.2 0.000566945
6.3 0.000616227
6.4 0.000658307
6.5 0.000693929
6.6 0.000723767
6.7 0.000748431
6.8 0.00076847
6.9 0.000784382
  7 0.000796615
7.1 0.000808881
7.2 0.000838594
7.3 0.000895834
7.4 0.000949404
7.5 0.000999355
7.6 0.00104575
7.7 0.00108865
7.8 0.00112815
7.9 0.00116433
  8 0.0011973
8.1 0.00122714
8.2 0.00125397
8.3 0.00127791
8.4 0.00129906

};
\addplot [semithick, color2, forget plot]
table {%
5.5 0.000137645
5.6 6.26773e-05
5.7 0.000127053
5.8 0.000182695
5.9 0.000230551
  6 0.000272032
6.1 0.000308916
6.2 0.000340198
6.3 0.000378375
6.4 0.000413847
6.5 0.000445389
6.6 0.000473434
6.7 0.000498349
6.8 0.000520445
6.9 0.000539988
  7 0.000557205
7.1 0.000572294
7.2 0.000585429
7.3 0.000596761
7.4 0.000606427
7.5 0.000614548
7.6 0.000621234
7.7 0.000626586
7.8 0.000630695
7.9 0.000633647
  8 0.000635522
8.1 0.000636393
8.2 0.000636329
8.3 0.000635397
8.4 0.000633658
};
\addplot [semithick, color3, forget plot]
table {%
5.5 4.5009e-06
5.6 2.15471e-05
5.7 4.26232e-05
5.8 5.96044e-05
5.9 7.32118e-05
  6 8.40382e-05
6.1 9.25717e-05
6.2 9.92143e-05
6.3 0.000104297
6.4 0.000108093
6.5 0.000110828
6.6 0.000112688
6.7 0.000113826
6.8 0.00011437
6.9 0.000114424
  7 0.000114075
7.1 0.000113395
7.2 0.000112444
7.3 0.000111271
7.4 0.000109917
7.5 0.000108417
7.6 0.000106799
7.7 0.000105086
7.8 0.000103298
7.9 0.000101451
  8 9.95607e-05
8.1 9.76372e-05
8.2 9.56905e-05
8.3 9.37289e-05
8.4 9.17595e-05
};
\addplot [semithick, color4, forget plot]
table {%
5.5 1.35672e-07
5.6 4.29705e-06
5.7 7.71841e-06
5.8 1.03305e-05
5.9 1.22956e-05
  6 1.37446e-05
6.1 1.47824e-05
6.2 1.54935e-05
6.3 1.59461e-05
6.4 1.61948e-05
6.5 1.62835e-05
6.6 1.62474e-05
6.7 1.61149e-05
6.8 1.59085e-05
6.9 1.56464e-05
  7 1.5343e-05
7.1 1.50097e-05
7.2 1.46556e-05
7.3 1.4288e-05
7.4 1.39123e-05
7.5 1.35331e-05
7.6 1.31537e-05
7.7 1.27768e-05
7.8 1.24042e-05
7.9 1.20376e-05
  8 1.1678e-05
8.1 1.13261e-05
8.2 1.09826e-05
8.3 1.06476e-05
8.4 1.03214e-05
};
\end{axis}

\end{tikzpicture}}
		\subcaption{}
		\label{fig:errors_bkwa05}
    \end{subfigure}
    \hfill
    \begin{subfigure}[b]{0.49\textwidth}
		\resizebox{0.925\linewidth}{!}{
\begin{tikzpicture}

\definecolor{color2}{rgb}{0.172549019607843,0.627450980392157,0.172549019607843}
\definecolor{color1}{rgb}{1,0.498039215686275,0.0549019607843137}
\definecolor{color4}{rgb}{0.580392156862745,0.403921568627451,0.741176470588235}
\definecolor{color0}{rgb}{0.12156862745098,0.466666666666667,0.705882352941177}
\definecolor{color3}{rgb}{0.83921568627451,0.152941176470588,0.156862745098039}

\begin{axis}[
xmin=5.355, xmax=8.54499999999999,
ymin=1e-9, ymax=0.01,
ymode=log,
tick align=inside,
tick pos=left,
x grid style={white!69.019607843137251!black},
y grid style={white!69.019607843137251!black},
ylabel = {$e_{L_2}(t)$},
xlabel = {time},
]

\addplot [semithick, color0, forget plot]
table {%
5.5 0.00539899
5.6 0.00492044
5.7 0.00456416
5.8 0.00431281
5.9 0.004148
  6 0.00405153
6.1 0.00400677
6.2 0.00399961
6.3 0.00401878
6.4 0.00405573
6.5 0.00410424
6.6 0.00415992
6.7 0.00421972
6.8 0.0042816
6.9 0.00434415
  7 0.00440644
7.1 0.00446784
7.2 0.00452788
7.3 0.00458624
7.4 0.00464267
7.5 0.00469694
7.6 0.00474888
7.7 0.00479831
7.8 0.00484509
7.9 0.00488908
  8 0.00493016
8.1 0.00496822
8.2 0.00500317
8.3 0.00503494
8.4 0.00506348
};
\addplot [semithick, color1, forget plot]
table {%
5.5 0.00104651
5.6 0.000802752
5.7 0.000801736
5.8 0.000977099
5.9 0.00121497
  6 0.00145874
6.1 0.00168837
6.2 0.00189742
6.3 0.00208441
6.4 0.00224982
6.5 0.00239491
6.6 0.00252121
6.7 0.00263033
6.8 0.00272381
6.9 0.00280315
  7 0.0028697
7.1 0.00292472
7.2 0.00296937
7.3 0.00300467
7.4 0.00303156
7.5 0.00305091
7.6 0.00306347
7.7 0.00306993
7.8 0.00307092
7.9 0.003067
  8 0.00305867
8.1 0.00304639
8.2 0.00303058
8.3 0.0030116
8.4 0.00298979
};
\addplot [semithick, color2, forget plot]
table {%
5.5 0.000193099
5.6 0.000170635
5.7 0.000340081
5.8 0.000511777
5.9 0.00066501
  6 0.000798792
6.1 0.000914466
6.2 0.00101375
6.3 0.00109835
6.4 0.00116985
6.5 0.00122968
6.6 0.00127916
6.7 0.00131943
6.8 0.00135155
6.9 0.00137645
  7 0.00139495
7.1 0.00140781
7.2 0.00141568
7.3 0.00141916
7.4 0.00141877
7.5 0.00141499
7.6 0.00140823
7.7 0.00139887
7.8 0.00138725
7.9 0.00137367
  8 0.00135838
8.1 0.00134163
8.2 0.00132363
8.3 0.00130457
8.4 0.00128461
};
\addplot [semithick, color3, forget plot]
table {%
5.5 6.05829e-06
5.6 3.99288e-05
5.7 7.67089e-05
5.8 0.000107345
5.9 0.000132681
  6 0.000153473
6.1 0.000170374
6.2 0.000183943
6.3 0.00019466
6.4 0.000202937
6.5 0.000209127
6.6 0.000213537
6.7 0.000216427
6.8 0.000218024
6.9 0.000218521
  7 0.000218088
7.1 0.00021687
7.2 0.00021499
7.3 0.000212559
7.4 0.000209668
7.5 0.000206398
7.6 0.00020282
7.7 0.000198993
7.8 0.000194969
7.9 0.000190793
  8 0.000186503
8.1 0.000182133
8.2 0.00017771
8.3 0.00017326
8.4 0.000168802
};
\addplot [semithick, color4, forget plot]
table {%
5.5 1.7915e-07
5.6 6.37937e-06
5.7 1.16077e-05
5.8 1.57749e-05
5.9 1.9064e-05
  6 2.16266e-05
6.1 2.35883e-05
6.2 2.50531e-05
6.3 2.61076e-05
6.4 2.68235e-05
6.5 2.72606e-05
6.6 2.74689e-05
6.7 2.74897e-05
6.8 2.73579e-05
6.9 2.71021e-05
  7 2.67467e-05
7.1 2.63118e-05
7.2 2.58141e-05
7.3 2.52678e-05
7.4 2.46846e-05
7.5 2.40742e-05
7.6 2.34447e-05
7.7 2.2803e-05
7.8 2.21546e-05
7.9 2.1504e-05
  8 2.08552e-05
8.1 2.02112e-05
8.2 1.95746e-05
8.3 1.89473e-05
8.4 1.83311e-05
};
\end{axis}

\end{tikzpicture}}
		\subcaption{}
		\label{fig:errors_bkwb05}		
    \end{subfigure}    
    \vspace{0.3cm}
	\begin{subfigure}[b]{0.8\textwidth}
	\begin{tikzpicture}

\definecolor{color2}{rgb}{0.172549019607843,0.627450980392157,0.172549019607843}
\definecolor{color1}{rgb}{1,0.498039215686275,0.0549019607843137}
\definecolor{color4}{rgb}{0.580392156862745,0.403921568627451,0.741176470588235}
\definecolor{color0}{rgb}{0.12156862745098,0.466666666666667,0.705882352941177}
\definecolor{color3}{rgb}{0.83921568627451,0.152941176470588,0.156862745098039}

\begin{axis}
[
hide axis,
xmin=-2.995, xmax=62.895,
ymin=5.53988011611066e-19, ymax=0.216926593216958,
ymode=log,
tick align=outside,
tick pos=left,
x grid style={lightgray!92.026143790849673!black},
legend cell align={left},
y grid style={lightgray!92.026143790849673!black},
legend columns=7,
legend entries={{$N=8$},{$N=12$},{$N=16$},{$N=24$},{$N=32$}},
legend style={at={(0.0,0.0)}, anchor=south west, draw=black}
]
\addlegendimage{no markers, color0}
\addlegendimage{no markers, color1}
\addlegendimage{no markers, color2}
\addlegendimage{no markers, color3}
\addlegendimage{no markers, color4}
\end{axis}
\end{tikzpicture}
	\end{subfigure}
	\caption{\subref{fig:errors_bkwa},\subref{fig:errors_bkwa075},\subref{fig:errors_bkwa05}: $L_\infty-$error $e_{L_\infty}(t)=\|f(t)-f_h(t)\|_{L_\infty}$, \subref{fig:errors_bkwb},\subref{fig:errors_bkwb075},\subref{fig:errors_bkwb05}: $L_2-$ error $e_{L_2}(t)=\|f(t)-f_h(t)\|_{L_\infty}$ over time for different polynomial orders and different numbers $\nip$ of integration nodes w.r.t. $\vbar$. \subref{fig:errors_bkwa}, \subref{fig:errors_bkwb}: $\nip = N$. \subref{fig:errors_bkwa075} and \subref{fig:errors_bkwb075}: $\nip = 0.75N$. \subref{fig:errors_bkwa05}, \subref{fig:errors_bkwb05}: $\nip = 0.5N$. Time stepping with a Runge Kutta 4 scheme $dt=0.1$.}
	\label{fig:errors_bkw}
\end{figure}

\begin{figure}
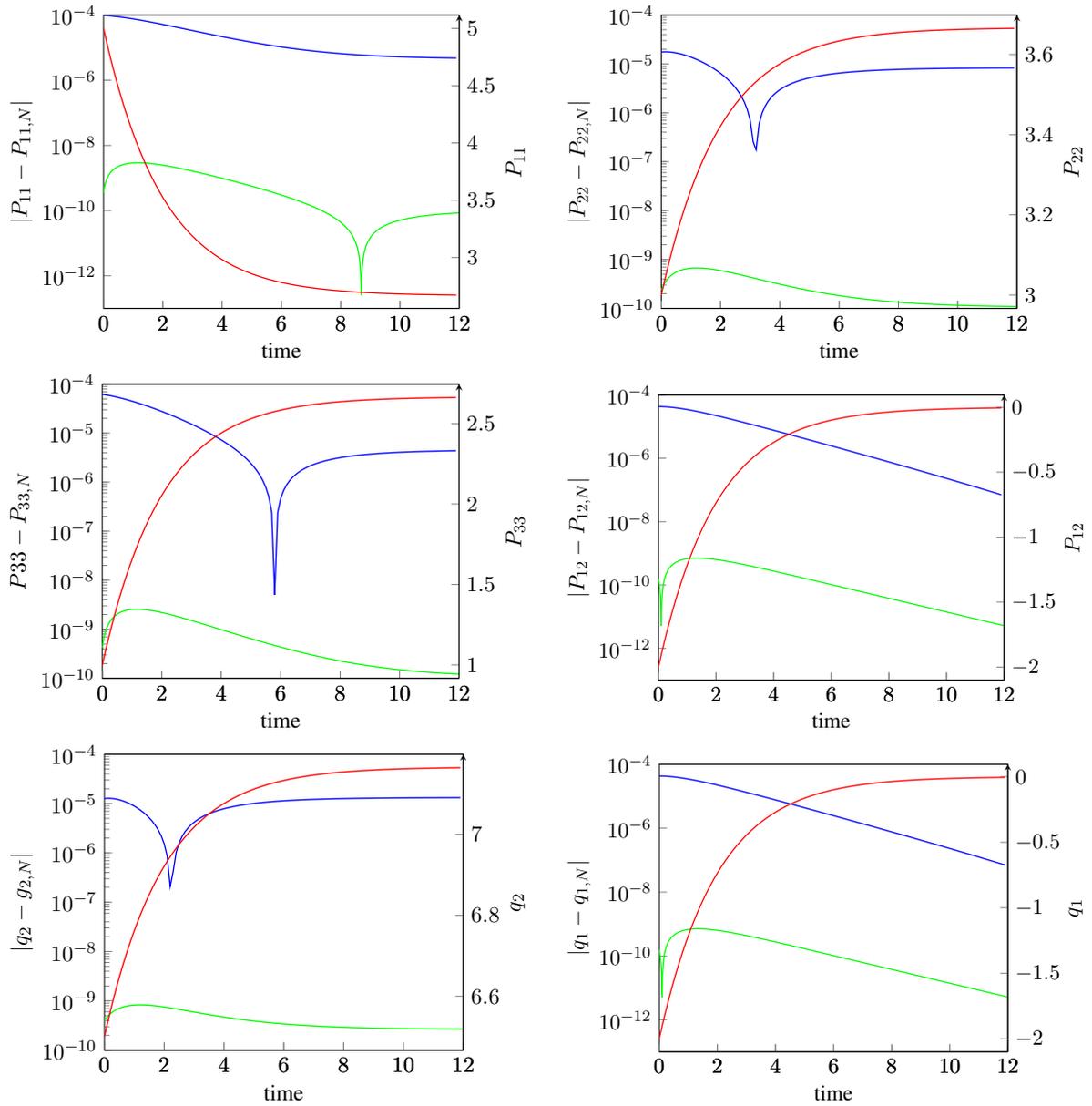

    \centering
    \begin{subfigure}[b]{0.49\textwidth}
		\resizebox{0.999\linewidth}{!}{\input{momentsp11.tikz}}
    \end{subfigure}
    \hfill
    \begin{subfigure}[b]{0.49\textwidth}
		\resizebox{0.999\linewidth}{!}{\input{momentsp22.tikz}}
    \end{subfigure}
    
    \begin{subfigure}[b]{0.49\textwidth}
		\resizebox{0.999\linewidth}{!}{\input{momentsp33.tikz}}
    \end{subfigure}
    \hfill
    \begin{subfigure}[b]{0.49\textwidth}
		\resizebox{0.999\linewidth}{!}{\input{momentsp12.tikz}}
    \end{subfigure}
    
    \begin{subfigure}[b]{0.49\textwidth}
		\resizebox{0.999\linewidth}{!}{\input{momentsq2.tikz}}
    \end{subfigure}
    \hfill
    \begin{subfigure}[b]{0.49\textwidth}
		\resizebox{0.999\linewidth}{!}{\input{momentsq1.tikz}}
    \end{subfigure}
	\caption{Second and third order moments (red) and difference to exact values (blue: $N=8$, green: $N=16$) for Maxwell molecules. Time stepping with a Runge Kutta 4-step method, $dt=0.1$ for $N=8$ and $dt=0.01$ for $N=16$.}
	\label{fig:moments_twopeaks_o8_10}	
\end{figure}

\begin{figure}
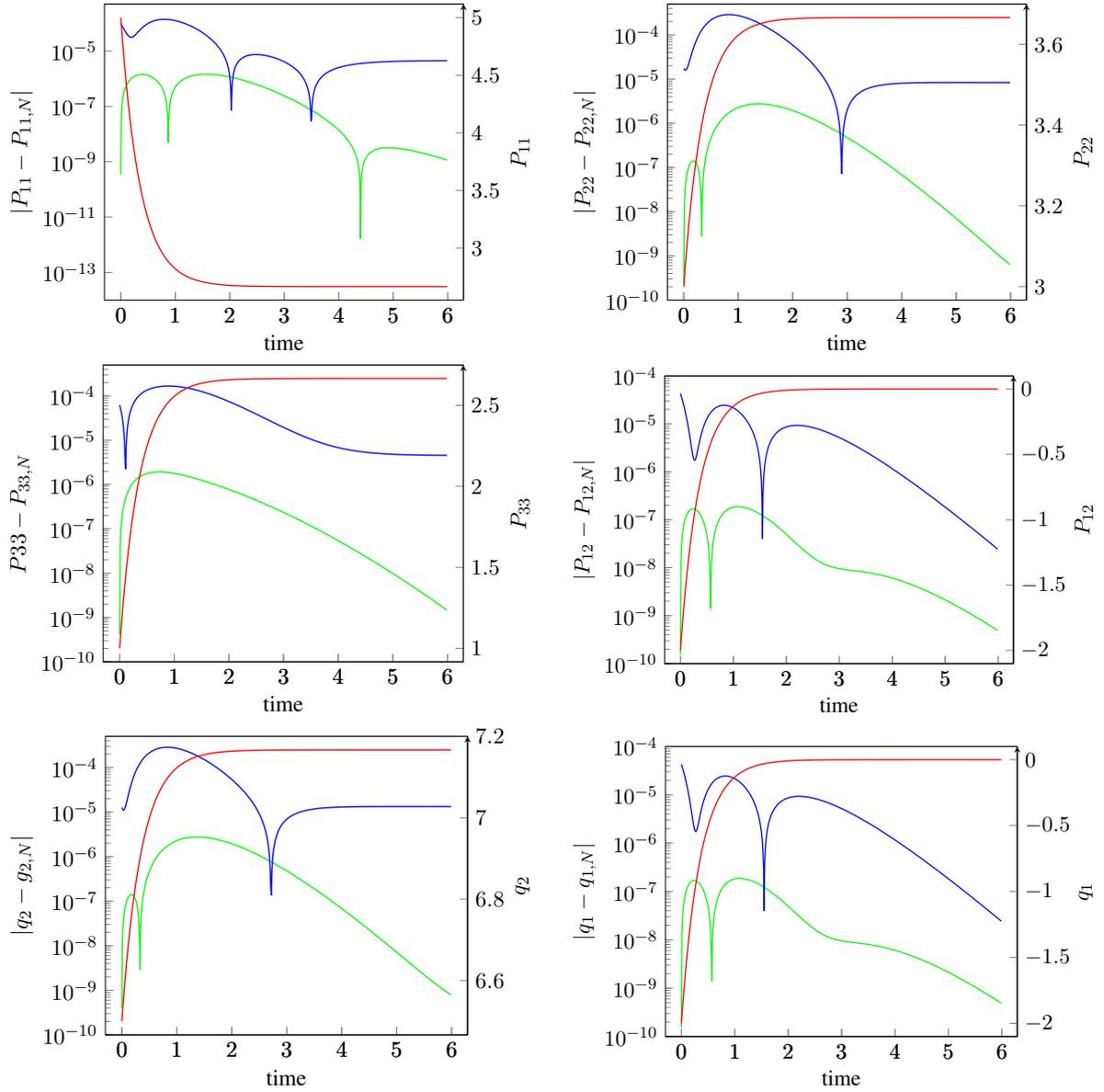

    \centering
    \begin{subfigure}[b]{0.49\textwidth}
		\resizebox{0.999\linewidth}{!}{\input{hsmomentsp11.tikz}}
    \end{subfigure}
    \hfill
    \begin{subfigure}[b]{0.49\textwidth}
		\resizebox{0.999\linewidth}{!}{\input{hsmomentsp22.tikz}}
    \end{subfigure}
    
    \begin{subfigure}[b]{0.49\textwidth}
		\resizebox{0.999\linewidth}{!}{\input{hsmomentsp33.tikz}}
    \end{subfigure}
    \hfill
    \begin{subfigure}[b]{0.49\textwidth}
		\resizebox{0.999\linewidth}{!}{\input{hsmomentsp12.tikz}}
    \end{subfigure}
    
    \begin{subfigure}[b]{0.49\textwidth}
		\resizebox{0.999\linewidth}{!}{\input{hsmomentsq2.tikz}}
    \end{subfigure}
    \hfill
    \begin{subfigure}[b]{0.49\textwidth}
		\resizebox{0.999\linewidth}{!}{\input{hsmomentsq1.tikz}}
    \end{subfigure}
	\caption{Time evolution of second and third order moments (red) for hard sphere molecules and their difference to a reference solution (blue: $N=8$, green: $N=16$). Time stepping with a Runge Kutta 4-step method, $dt=0.01$. The reference solution was calculated with $N=22$ and $dt=0.001$.}
	\label{fig:moments_twopeaks_hardspheres_o8}
\end{figure}

\begin{figure}
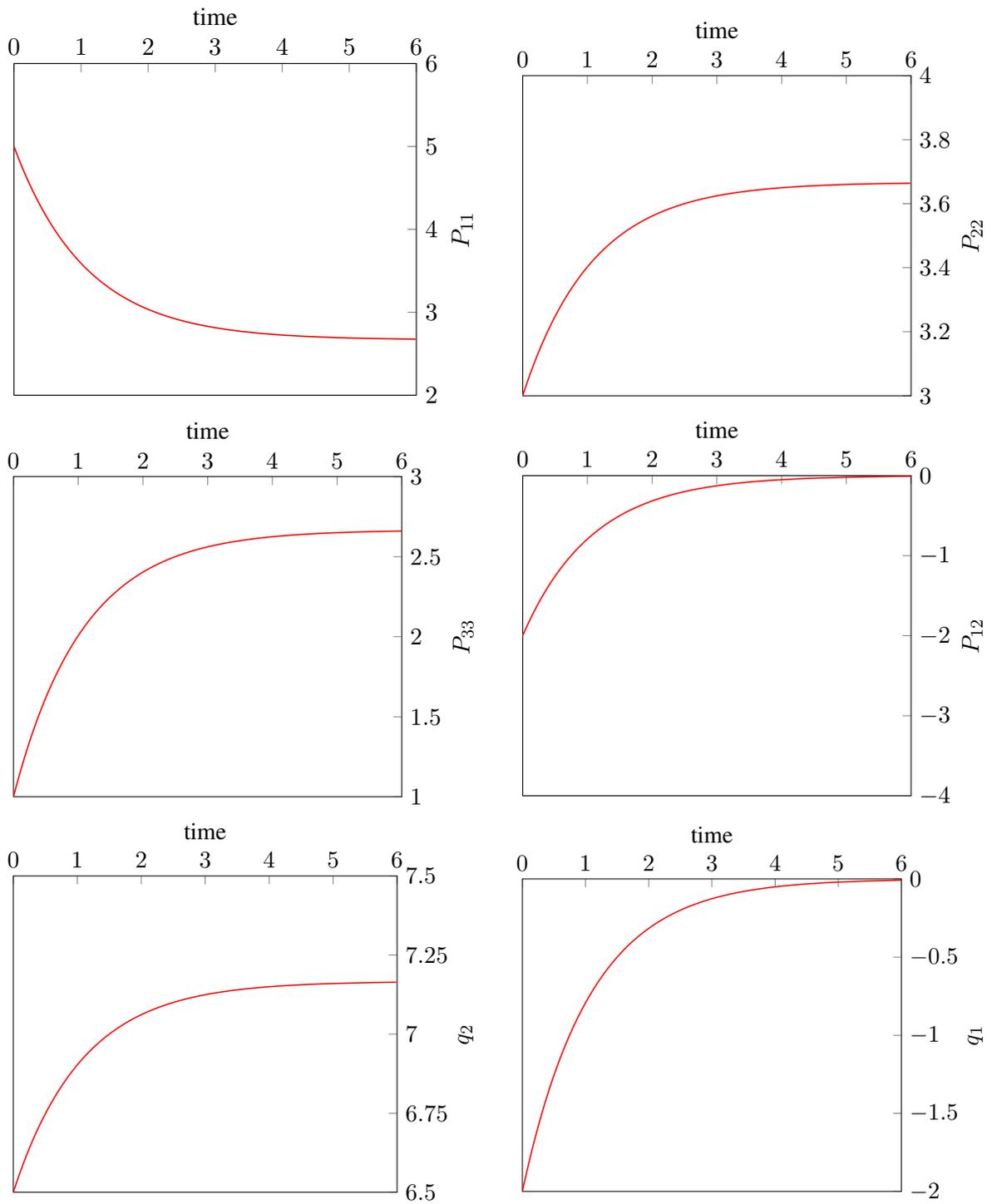

    \centering
    \begin{subfigure}[b]{0.49\textwidth}
		\resizebox{0.999\linewidth}{!}{\input{armomentsp11.tikz}}
    \end{subfigure}
    \hfill
    \begin{subfigure}[b]{0.49\textwidth}
		\resizebox{0.999\linewidth}{!}{\input{armomentsp22.tikz}}
    \end{subfigure}
    
    \begin{subfigure}[b]{0.49\textwidth}
		\resizebox{0.999\linewidth}{!}{\input{armomentsp33.tikz}}
    \end{subfigure}
    \hfill
    \begin{subfigure}[b]{0.49\textwidth}
		\resizebox{0.999\linewidth}{!}{\input{armomentsp12.tikz}}
    \end{subfigure}
    
    \begin{subfigure}[b]{0.49\textwidth}
		\resizebox{0.999\linewidth}{!}{\input{armomentsq2.tikz}}
    \end{subfigure}
    \hfill
    \begin{subfigure}[b]{0.49\textwidth}
		\resizebox{0.999\linewidth}{!}{\input{armomentsq1.tikz}}
    \end{subfigure}
	\caption{Time evolution of second and third order moments for argon molecules. Numerical solution with $N=16$. Time stepping with a Runge Kutta 4-step method, $dt=0.01$.}
	\label{fig:moments_twopeaks_o16_argon}
\end{figure}

\FloatBarrier

\bibliographystyle{unsrt}

\end{document}